\documentclass[10pt]{article}
\usepackage[a4paper,centering,vscale=0.75,hscale=0.7]{geometry}
\usepackage{mathtools,
  amssymb,
  amsthm,
  amsfonts}
\usepackage{mathrsfs, color}

\newtheorem{thm}{Theorem}[section]
\newtheorem{lemma}[thm]{Lemma}
\newtheorem{prop}[thm]{Proposition}

\theoremstyle{definition}
\newtheorem{definition}[thm]{Definition}
\theoremstyle{remark}
\newtheorem{rmk}[thm]{Remark}
\newtheorem{example}[thm]{Example}

\newcommand{\E}{\mathop{{}\mathbb{E}}\nolimits}
\newcommand{\cF}{\mathscr{F}}
\newcommand{\bG}{\mathbb{G}}
\newcommand{\bL}{\mathbb{L}}
\newcommand{\cL}{\mathscr{L}}
\renewcommand{\P}{\mathbb{P}}
\newcommand{\erre}{\mathbb{R}}
\newcommand{\bS}{\mathbb{S}}
\newcommand{\eps}{\varepsilon}

\newcommand{\diff}[2]{{#2}^{(#1)}}

\DeclarePairedDelimiter\abs{\lvert}{\rvert}
\DeclarePairedDelimiter\norm{\lVert}{\rVert}
\DeclarePairedDelimiterX\ip[2]{\langle}{\rangle}{#1,#2}

\numberwithin{equation}{section}
\numberwithin{thm}{section}

\title{Fr\'echet differentiability of mild solutions to SPDEs with
  respect to the initial datum}
\author{Carlo Marinelli${}^1$ and Luca Scarpa${}^2$\\[6pt]
  \scriptsize ${}^1$Department of Mathematics, University College London,
  Gower Street, London WC1E 6BT, UK.\\
  \scriptsize ${}^2$Fakult\"at f\"ur Mathematik, Universit\"at Wien,
  Oskar-Morgenstern-Platz 1, 1090 Wien, Austria.}
\date{\normalsize August 2, 2019}

\begin{document}
\maketitle

\newif\ifbozza
\bozzafalse

\begin{abstract}
  We establish $n$-th order Fr\'echet differentiability with respect
  to the initial datum of mild solutions to a class of jump-diffusions
  in Hilbert spaces. In particular, the coefficients are
  Lipschitz continuous, but their derivatives of order higher than one
  can grow polynomially, and the (multiplicative) noise sources are a
  cylindrical Wiener process and a quasi-left-continuous
  integer-valued random measure.  As preliminary steps, we prove
  well-posedness in the mild sense for this class of equations, as
  well as first-order G\^ateaux differentiability of their solutions
  with respect to the initial datum, extending previous results
  by Marinelli, Pr\'ev\^ot, and R\"ockner in
  several ways.  The differentiability results obtained here are a
  fundamental step to construct classical solutions to non-local
  Kolmogorov equations with sufficiently regular coefficients by
  probabilistic means.
\end{abstract}

\section{Introduction}
Our goal is to obtain existence and uniqueness of mild solutions, and,
especially, their differentiability with respect to the initial datum,
to a class of stochastic evolution equations on Hilbert spaces of the
form
\begin{equation}
  \label{eq:1}
  \left\{
    \begin{aligned}
      &du(t) + A u(t)\,dt = f(t,u(t))\,dt + B(t,u(t))\,dW(t) + \int_Z
      G(t,z,u(t-))\,\bar{\mu}(dt,dz),\\
      &u(0)=u_0.
    \end{aligned}
  \right.
\end{equation}
Here $A$ is a linear $m$-accretive operator, $W$ is a cylindrical
Wiener process, $\bar{\mu}$ is a compensated integer-valued
quasi-left-continuous random measure, and the coefficients $f$, $B$,
$G$ satisfy suitable measurability and Lipschitz continuity
conditions. Precise assumptions on the data of the problem are stated
in Sections~\ref{ssec:not} and \ref{sec:WP} below.

The results extend (and partially supersede) those obtained in
\cite{cm:JFA10} in several ways: (a) well-posedness is established
here in much greater generality, in particular allowing $\bar{\mu}$ to
be a quite general random measure, rather than just a compensated
Poisson measure as in \cite{cm:JFA10}. Moreover, using a more precise
maximal estimate for stochastic convolutions, solutions are no longer
needed to be sought in spaces of processes with finite second moment
(yet more general well-posedness results are going to appear in
\cite{cm:EqLp}); (b) the sufficient conditions on the coefficients of
\eqref{eq:1} for the differentiability of its solution with respect to
the initial datum are the natural ones. For instance, roughly
speaking, Fr\'echet differentiability of $f$, $B$, and $G$ imply
Fr\'echet differentiability of the solution map $u_0 \mapsto u$, while
in \cite{cm:JFA10} a $C^1$ condition on $f$, $B$, and $G$ was
needed. In fact, the proof in \cite{cm:JFA10} was based on an implicit
function theorem with parameters, for which the $C^1$ assumption seems
indispensable, while here we use a direct approach based on the
definition of derivative; (c) we study the $n$-th order
differentiability of the solution map for arbitrary natural $n$,
instead of considering only first and second-order differentiability
as in \cite{cm:JFA10}. In this regard it is worth mentioning that we
just assume that the derivatives of $f$, $B$, and $G$ of order higher
than one satisfy a polynomial growth condition. While this assumption
causes non-trivial technical difficulties, it is more natural than
much more restrictive boundedness conditions that are often found in
the literature: a possible example of coefficients with 
nonbounded higher derivatives is given in Example~\ref{ex:Fn} below.

There are several reasons to study the differentiability of solutions
to stochastic equations in infinite dimensions with respect to the
initial datum (or, more generally, with respect to parameters), among
which the probabilistic construction of solutions to Kolmogorov
equations is our main motivation. This vast and mature field of
investigation is still very active, especially regarding stochastic
equations with additive Wiener noise: see, e.g., \cite{Krylov:diff}
for classical results in the finite-dimensional case, \cite{DPZ} for
basic results in the Hilbertian setting, and
\cite{Bog:K,Cerrai:LNM,DP:K,Stannat:rev} for accounts of more recent
developments. On the other hand, the case of equations with
discontinuous noise, for which the associated Kolmogorov equations are
of non-local type, is much less investigated, especially in the
infinite-dimensional setting (see \cite{cm:JFA10} for simple results
and \cite{cm:DPDE10} for a special case). As an application of the
above-mentioned differentiability results, we shall construct, in a
forthcoming work, classical solutions to non-local Kolmogorov
equations with sufficiently regular coefficients. As is well known,
such results are essential to consider Kolmogorov equations motivated
by applications, that usually have less regular coefficients. In fact,
a typical approach is, roughly speaking, to regularize the
coefficients of the equation, thus obtaining a family of approximating
Kolmogorov equations that are sufficiently simple to have classical
solutions, and to obtain a solution to the original problem passing to
the limit, in an appropriate sense, with respect to the regularization
parameter.  In this spirit, our ultimate goal is the extension of the
results in \cite{cm:POTA19} to non-local Kolmogorov equations
associated to stochastic evolution equations with jumps in a
generalized variational setting as considered in \cite{cm:semimg}.

Since the literature on the problem at hand is very large, it is not
easy to provide an accurate comparison of our results with existing
ones, apart of the remarks already made. We should nonetheless mention
the recent work \cite{Jentzen}, which considers a problem
analogous to ours, but without discontinuous noise term and with
coefficients with bounded derivatives of all orders. Here, the authors
exploit the smoothing property of an analytic semigroup 
and study differentiability in negative order spaces.

The remaining text is organized as follows: in \S\ref{sec:prel}, after
fixing some notation, we recall a characterization of G\^ateaux and
Fr\'echet differentiability, as well as some maximal estimates for
deterministic and stochastic convolutions, all of which are essential
tools. Well-posedness of \eqref{eq:1}, i.e. existence and uniqueness
of a mild solution and its continuous dependence on the initial datum,
is proved in \S\ref{sec:WP}. The remaining sections are devoted to
differentiability properties of the mild solution to \eqref{eq:1} with
respect to the initial datum: first-order G\^ateaux and Fr\'echet
differentiability are treated in \S\ref{sec:G} and \S\ref{sec:F},
respectively, and $n$-th order Fr\'echet differentiability is
considered in \S\ref{sec:F+}.

\medskip

\noindent\textbf{Acknowledgements.} A large part of the work for this
paper was done during several visits of the first-named author to the
Interdiszplin\"ares Zentrum f\"ur Komplexe Systeme (IZKS) at the
University of Bonn, Germany, and a visit to the University of Vienna,
Austria.  The warm hospitality of his hosts (S.~Albeverio and
U.~Stefanelli, respectively) and the good working conditions are
gratefully acknowledged. 
The second-named author is
funded by Vienna Science and Technology Fund 
(WWTF) through Project MA14-009.
The authors are indebted to G.~Luise for
contributing to a preliminary draft.


\ifbozza\newpage\else\fi
\section{Preliminaries}
\label{sec:prel}
\subsection{Notation}
\label{ssec:not}
The spaces of linear bounded operators from a Banach space $E$ to a
further Banach space $F$ will be denoted by $\cL(E, F)$, and
$\cL^2(E,F)$ stands for the space of Hilbert-Schmidt operators from $E$
to $F$ if $E$ and $F$ are Hilbert spaces.  The closed ball of radius
$r>0$ in $E$ will be denoted by $B_r(E)$.

All stochastic elements will be defined on a fixed filtered
probability space $(\Omega,\mathcal{F},\mathbb{F},\P)$, with the
filtration $\mathbb{F}:=(\mathcal{F}_{t})_{t\in [0,T]}$ complete and
right-continuous, and $T>0$ a fixed final time.
Moreover, $H$ will always denote a fixed real separable Hilbert space
with norm $\norm{\cdot}$.
For any $p>0$ and $[t_0,t_1] \subseteq [0,T]$, we shall use the
notation $\bS^p(t_0, t_1)$ for the space of adapted c\`adl\`ag
$H$-valued processes $Y$ such that
\[
  \norm[\big]{Y}_{\bS^p(t_0,t_1)} := \Bigl( \E\sup_{t\in[t_0,
    t_1]}\norm{Y(t)}^p \Bigr)^{1/p} < +\infty,
\]
and we set $\bS^p:=\bS^p(0,T)$. We recall that these are Banach spaces
if $p \geq 1$, and quasi-Banach spaces if
$p \in \mathopen]0,1\mathclose[$. In the latter case the triangle
inequality is reversed, but one has
\[
  \norm[\big]{Y_1+Y_1}_{\bS^p(t_0, t_1)} \leq 2^{1/p} \bigl(
  \norm[\big]{Y_1}_{\bS^p(t_0, t_1)} + \norm[\big]{Y_2}_{\bS^p(t_0,
    t_1)} \bigr),
\]
to which we shall also refer, with a harmless abuse of terminology, as
the triangle inequality. Moreover, $\bS^p(t_0, t_1)$ is a complete
metric space for every $p>0$ when endowed with the distance
\[
  d_{p,t_0,t_1}(Y_1,Y_2) := \norm[\big]{Y_1 - Y_2}_{\bS^p(t_0, t_1)}^{1\wedge
    p},
\]
as it follows from the inequality
$\abs{x+y}^p \leq \abs{x}^p + \abs{y}^p$, which holds true for every
$x$, $y \in \erre$ and $p \in \mathopen]0,1\mathclose[$.
For brevity we shall write $d_p:=d_{p,0,T}$. Entirely
analogously, $L^p(\Omega;H)$ endowed with the distance
\[
  (Y_1,Y_2) \mapsto \norm[\big]{Y_1 - Y_2}_{L^p(\Omega;H)}^{1\wedge p}
\]
is a complete metric space for every $p>0$.

\medskip

Let $K$ be a real separable Hilbert space and $W$ a cylindrical Wiener
process on $K$.
Let $(Z,\mathscr{Z})$ be a Blackwell measurable space and $\mu$ an
integer-valued quasi-left-continuous random measure on
$Z \times [0,T]$, independent of $W$, with dual predictable projetion
(compensator) $\nu$, and $\bar{\mu}:=\mu-\nu$. 
We recall that the assumption on $(Z,\mathscr{Z})$ as a Blackwell
space is usually required in the literature on random measures
(see \cite[\S1a]{JacShi}), and it ensures for example that $\mathscr Z$
is separable and generated by a countable algebra.
We also recall that the
quasi-left-continuity of $\mu$ implies that the random measure $\nu$
is non-atomic (see, e.g., \cite[Corollary~1.19, p.~70]{JacShi}). A map
$g: \Omega \times [0,T] \times Z \to H$ will be called predictable if
it is $\mathscr{P} \otimes \mathscr{Z}$-measurable, where
$\mathscr{P}$ stands for the predictable $\sigma$-algebra of
$\mathbb{F}$ (the target space $H$ is always assumed to be endowed
with the Borel $\sigma$-algebra). Moreover, for any such predictable
map $g$, we set, for any $p$, $q \in \mathopen]0,\infty\mathclose[$,
\[
  \norm[\big]{g}_{L^q(\nu;H)} := \biggl(
  \int_{\mathopen]0,T\mathclose] \times Z} \norm{g}^q\,d\nu
  \biggr)^{1/q},%
  \qquad%
  \norm[\big]{g}_{L^p(\Omega;L^q(\nu;H))} := \biggl( \E\biggl(
  \int_{\mathopen]0,T\mathclose] \times Z} \norm{g}^q\,d\nu
  \biggr)^{p/q} \biggr)^{1/p}
\]
and
\[
  \norm[\big]{g}_{\bG^p} :=
  \begin{cases}
    \norm[\big]{g}_{L^p(\Omega;L^2(\nu;H))} \qquad
    &\text{if } p \in \mathopen]0,1],\\[4pt]
    \displaystyle \inf_{g_1+g_2=g} \bigl( %
    \norm[\big]{g_1}_{L^p(\Omega;L^2(\nu;H))} +
    \norm[\big]{g_2}_{L^p(\Omega;L^p(\nu;H))} \bigr)
    &\text{if }  p \in \mathopen]1,2\mathclose[,\\[4pt]
    \norm[\big]{g}_{L^p(\Omega;L^2(\nu;H))} +
    \norm[\big]{g}_{L^p(\Omega;L^p(\nu;H))} &\text{if } p \in
    [2,\infty\mathclose[,
  \end{cases}
\]
where the infima are taken with respect to
$\mathscr{P} \otimes \mathscr{Z}$-measurable maps $g_1$, $g_2$ only.
One may actually show that $L^p(\Omega;L^q(\nu;H))$ as well as $\bG^p$
are (quasi-)Banach space and that
\[
  \bG^p =
  \begin{cases}
    L^p(\Omega;L^2(\nu;H)), & p \in ]0,1],\\[4pt]
    L^p(\Omega;L^2(\nu;H)) + L^p(\Omega;L^p(\nu;H)), & p \in [1,2],\\[4pt]
    L^p(\Omega;L^2(\nu;H)) \cap L^p(\Omega;L^p(\nu;H)), & p \in
    [2,\infty[.
  \end{cases}
\]
For a proof of this statement, as well as of other properties of such
mixed-norm $L^p$ spaces involving random measures (even in a more
general setting), we refer to \cite{DY:dual}. For us, however, it is
enough to know that they are quasi-normed spaces, and the ``norms''
just introduced on spaces where the underlying measure is random is
only a convenient notation. We shall also need to consider spaces
where $\mathopen]0,T\mathclose] \times Z$ is replaced by
$\mathopen]t_0,t_1\mathclose] \times Z$, with
$0 \leq t_0 \leq t_1 \leq T$, and the corresponding notation will be
self-explanatory.

We shall use standard notation of stochastic calculus: we write, for
instance, $f^*$ and $f_-$ to denote the maximal function
and the left-limit function of a c\`adl\`ag
function $f$, respectively.
Further notation related to deterministic and stochastic
convolutions, as well as to different notions of derivative for maps
between infinite-dimensional spaces, will be introduced where they
first appear. For any $a,b>0$, we use the notation $a\lesssim b$
to indicate that there exists a constant $c>0$ such that $a\leq cb$.
If $c$ depends on some further quantities that we need to keep track of
we shall indicate them in a subscript.
We use the classical notation $\wedge$ and $\vee$ for $\min$ and
$\max$, respectively.

\subsection{Notions of derivative}
\label{ssec:der}
Let $E$, $F$ be Banach spaces, and $G$ be a subspace of $E$.  A
function $\phi: E \to F$ is G\^ateaux differentiable at $x_0 \in E$
along $G$ if there exists a continuous linear map $L \in \cL(G,F)$
such that
\[
  \lim_{\varepsilon \to 0} \frac{\phi(x_0+\varepsilon
    h)-\phi(x_0)}{\varepsilon} = Lh \qquad \forall h \in G.
\]
The linear map $L$, which is necessarily unique, will be denoted by
$D_{\mathcal G}\phi(x_0)$ and is called the G\^ateaux derivative of $\phi$ at
$x_0$ (along the subspace $G$, if $G \neq E$). If $G=E$ and 
$\phi$ is also Lipschitz continuous with Lipschitz constant $L_\phi$, 
it easily follows from the definition that 
$\norm{D_{\mathcal G}\phi(x_0)}_{\cL(E,F)}\leq L_\phi$: indeed,
for all $h\in E$ we have
\[
  \norm{D_{\mathcal G}\phi(x_0)h}_F=\lim_{\eps\to0}\frac{\norm{\phi(x_0+\eps h) - \phi(x_0)}_F}{\eps}\leq
  L_\phi\frac{\norm{x_0+\eps h - x_0}_E}\eps=L_\phi\norm{h}_E.
\]

The map $\phi$ is Fr\'echet differentiable at $x_0 \in E$ along the
subspace $G$ if there exists a continuous linear map $L \in \cL(G,F)$
such that
\[
  \lim_{h \to 0} \frac{\phi(x_0+h) - \phi(x_0) - Lh}{\norm{h}_G}=0.
\]
The (unique) map $L$ will be denoted by $D\phi(x_0)$ and is called the
Fr\'echet derivative of $\phi$ at $x_0$ (along the subspace $G$, in
case $G \neq E$). It is well known that Fr\'echet differentiability implies
G\^ateaux differentiability, while the converse is not true.
We shall often use the following characterization of Fr\'echet
differentiability, of which we include a proof for the convenience of
the reader.
\begin{lemma}
  \label{lm:Frech}
  A map $\phi:E \to F$ is Fr\'echet differentiable at $x_0\in E$ with
  $D\phi(x_0)=L$ if and only if for each bounded set $B \subset E$ one
  has
  \begin{equation}
    \label{equiv:cond}
    \lim_{\varepsilon \to 0}
    \frac{\phi(x_0+\varepsilon h)-\phi(x_0) - \varepsilon Lh}%
    {\varepsilon} = 0
  \end{equation}
  uniformly with respect to $h \in B$.
\end{lemma}
\begin{proof}
  Let $\phi$ be Fr\'echet differentiable at $x_0$ with $D\phi(x_0)=L$,
  and set $R(h):=\phi(x_0+h)-\phi(x_0)-Lh$. Then $R(h)/\norm{h} \to 0$
  as $h \to 0$. Let $B$ be a bounded set and $M$ a real number such
  that $B$ is included in the ball of $E$ of radius $M$ centered at
  zero. For any $\eta>0$ there exists $\delta>0$ such that
  $\norm{R(h)}/\norm{h} \leq \eta/M$ for every $h$ with
  $\norm{h} \leq \delta$. Therefore, for any $\varepsilon$ such that
  $\abs{\varepsilon}\leq \delta/M$, one has
  $\norm{\varepsilon h} \leq \delta$ and
  \[
    \norm{R(\varepsilon h)}\leq \eta \frac{\norm{\varepsilon
        h}}{M}=\eta \abs{\varepsilon}\frac{\norm{h}}{M} \leq \eta
    \abs{\varepsilon},
  \]
  i.e. $\norm{R(\varepsilon h)}/\abs{\varepsilon} \to 0$ as
  $\varepsilon \to 0$ uniformly with respect to $h \in B$.
  Let us now prove the converse implication: assume that
  \eqref{equiv:cond} holds for every $B$, uniformly with respect to
  $h \in B$, and that, by contradiction, $\phi$ is not Fr\'echet
  differentiable at $x_0$, i.e. that $R(h)/\norm{h}$ does not converge
  to zero as $h \to 0$. In particular, there exists a sequence
  $(k_n) \subset E \setminus \{0\}$ converging to zero such that
  $R(k_n)/ \norm{k_n}$ does not converge to zero. We claim that it
  cannot happen that
  \[
    \sup_{h \in B} \frac{\phi(x_0+\varepsilon h)-\phi(x_0) -
      \varepsilon Lh} {\varepsilon} \longrightarrow 0
  \]
  as $\varepsilon \to 0$. In fact, setting
  $\varepsilon_n := \norm{k_n}$, $h_n=k_n/\norm{k_n}$, and $B:=(h_n)$,
  this would imply that
  $\varepsilon_n^{-1} \bigl( \phi(x_0 + \varepsilon_n h_n) -
  \varphi(x_0) - \varepsilon_n Lh_n \bigr)$ converges to zero as
  $n \to \infty$, which is equivalent to $R(k_n)/ \norm{k_n} \to 0$.
\end{proof}

By a simple scaling argument it is evident that it is sufficient to
consider as bounded subset $B$ the unit ball in $E$. One can thus say
that $\phi: E \to F$ is Fr\'echet differentiable at $x_0 \in E$ along
a subspace $G \subseteq E$ if there exists a continuous linear map
$L:G \to F$ such that
\[
  \lim_{\varepsilon \to 0} %
  \frac{\phi(x_0+\varepsilon h)-\phi(x_0) - \varepsilon Lh}%
  {\varepsilon} = 0 \qquad \text{uniformly on }\{h\in G:\text{
  }\norm{h}_G\leq 1.\}
\]

For a comprehensive treatment of differential calculus for functions
between topological vector spaces we refer to \cite{AmbPro}
for basic results in the case of Banach spaces, and to
\cite{AveSmol:diff, Bog:TVS} for the general case.

\subsection{Estimates for deterministic and stochastic convolutions}
\label{ssec:maxi}
Throughout this section $S$ stands for a strongly continuous linear
semigroup of contractions on $H$, and $-A$ for its generator.
Clearly, $A$ is necessarily a linear maximal monotone operator.

Here and in the following we shall use $S \ast g$ to denote
convolution of $S$ and an $H$-valued measurable function $g$ on
$\erre_+$, defined as
\[
  S \ast g: \erre_+ \ni t \longmapsto \int_0^t S(t-s)g(s)\,ds,
\]
under the minimal assumption that $S(t-\cdot)g \in L^1(0,t;H)$ for all
$t$ in a set of interest, usually a bounded interval of $\erre_+$.

The following estimate for convolutions is trivial, but sufficient for
our purposes.
\begin{lemma}
  \label{lm:dc}
  For every $p>0$ and for every measurable adapted process
  $\phi: \Omega \times [0,T] \to H$ such that
  $\phi \in L^p(\Omega; L^1(0,T; H))$,
  it holds that $S\ast\phi\in \bS^p(0,T)$ and
  \[
    \norm[\big]{S \ast \phi}_{\bS^p(0,T)} \leq
    \norm[\big]{\phi}_{L^p(\Omega;L^1(0,T;H))}.
  \]
\end{lemma}
\begin{proof}
  Minkowski's inequality and contractivity of $S$ immediately yield
  \[
    \E\sup_{t\leq T} \norm[\bigg]{\int_0^t S(t-s)\phi(s)\,ds}^p \leq
    \E\biggl( \sup_{t\leq T} \int_0^t \norm[\big]{S(t-s)\phi(s)}\,ds
    \biggr)^p \leq \E\biggl( \int_0^T \norm[\big]{\phi(s)}\,ds
    \biggr)^p. \qedhere
  \]
\end{proof}

We shall also need estimates for stochastic convolutions with respect
to the cylindrical Wiener process $W$, for which we shall always use
the following notation: for any $\cL^2(K,H)$-valued process $G$, the
stochastic convolution $S \diamond G$ is the process defined as
\[
  S \diamond G(t) := \int_0^t S(t-s)G(s)\,dW(s), \qquad t \geq 0,
\]
under a stochastic integrability assumption on $S(t-\cdot)G$.
There is an extensive literature on maximal estimates for stochastic
convolutions, mostly obtained through the so-called factorization
method by Da Prato, Kwapie\'n, and Zabczyk \cite{DPKZ}, which requires
$-A$ to generate a holomorphic semigroup. The following estimate
instead requires $A$ to be maximal monotone and can be proved by
relatively elementary techniques of stochastic calculus (see, e.g.,
\cite{cm:SIMA18} for a proof in a more general context).
\begin{prop}
  \label{prop:sc}
  For every $p>0$ and for every
  $G \in L^p(\Omega;L^2(0,T;\cL^{2}(K,H)))$ progressively
  measurable, the sto\-chas\-tic convolution $S \diamond G$ admits a
  modification in $\bS^p(0,T)$ and
  \[
    \norm[\big]{S \diamond G}_{\bS^p(0,T)} \lesssim_p
    \norm[\big]{G}_{L^p(\Omega;L^2(0,T;\cL^2(K,H)))}.
  \]
\end{prop}

Finally, a key role is played by the following maximal estimate for
stochastic convolutions with respect to the compensated random measure
$\bar{\mu}$. For a predictable $H$-valued process $g$, the stochastic
convolution of $g$ with respect to $\bar{\mu}$ will be denote by
$S \diamond_\mu g$ and defined as
\[
  S \diamond_\mu g(t) := \int_{\mathopen]0,t\mathclose]}\!\int_Z
  S(t-s) g(s,z) \,\bar{\mu}(ds,dz), \qquad t \geq 0,
\]
under a stochastic integrability assumption on $S(t-\cdot)g$ with
respect to $\bar{\mu}$.
\begin{lemma}
  \label{prop:scs}
  For every $p>0$ and for every $g \in \bG^p$, the stochastic convolution
  $S \diamond_\mu g$ admits a c\`adl\`ag modification and
  \[
    \norm[\big]{S \diamond_\mu g}_{\bS^p} \lesssim
    \norm[\big]{g}_{\bG^p}.
  \]
\end{lemma}
\noindent A proof can be found in \cite{cm:SemProba14}. A
generalization of this inequality to $L^q$-valued processes
will appear in \cite{cm:EqLp}.


\ifbozza\newpage\else\fi
\section{Well-posedness}
\label{sec:WP}
This section is devoted to the proof of well-posedness of equation
\eqref{eq:1}. We show existence and uniqueness of a mild solution, as
well as its continuous dependence on the initial datum, in spaces of
processes with finite moments of order
$p \in \mathopen]0,+\infty\mathclose[$. Although only the case
$p \geq 1$ is needed in the following sections on differentiability of
the solution with respect to the initial datum, the general case $p>0$
is necessary to deal with initial data or driving random measures
admitting finite moments of order strictly less than one. An example
is given by $\alpha$-stable random measures with $\alpha<1$.

The following assumptions (A0)--(A4) on the coefficients and the initial
datum of \eqref{eq:1} are in force throughout the paper.
\begin{itemize}
\item[(A0)] The initial datum $u_0$ is an $\cF_0$-measurable random
  variable with values in $H$;
\item[(A1)] $A$ is a linear maximal monotone operator on $H$, and $S$
is the strongly continuous semigroup of contractions generated by $-A$ on $H$;
\item[(A2)] The function $f: \Omega \times [0,T] \times H \to H$ is
  such that $f(\cdot,\cdot,x)$ is measurable and adapted for every
  $x \in H$, and there exists a constant $C_f>0$ such that
  \begin{align*}
    \norm{f(\omega,t,x)} &\leq C_f \bigl( 1 + \norm{x}\bigr),\\
    \norm{f(\omega,t,x)-f(\omega,t,y)} &\leq C_f \norm{x-y}
  \end{align*}
  for all $\omega \in \Omega$, $t \in [0,T]$, and $x,\,y \in H$;
\item[(A3)] The function
  $B: \Omega \times [0,T] \times H \to \cL^2(K,H)$ is such that
  $B(\cdot,\cdot,x)$ is progressively measurable for all $x \in H$,
  and there exists a constant $C_B>0$ such that
  \begin{align*}
    \norm[\big]{B(\omega,t,x)}_{\cL^2(K,H)}
    &\leq C_B \bigl( 1 + \norm{x}\bigr),\\
    \norm[\big]{B(\omega,t,x) - B(\omega,t,y)}_{\cL^2(K,H)}
    &\leq C_B \norm{x-y}
  \end{align*}
  for all $\omega \in \Omega$, $t \in [0,T]$, and $x,\,y \in H$;
\item[(A4)] The function
  $G:\Omega \times [0,T] \times Z \times H \to H$ is such that
  $G(\cdot,\cdot,\cdot,x)$ is
  $\mathscr{P} \otimes \mathscr{Z}$-measurable for all $x \in
  H$. Moreover,
  \begin{itemize}
  \item[(i)] if $p\leq 1$ or $p \geq 2$, then there exists a
    $\mathscr{P} \otimes \mathscr{Z}$-measurable function
    $g:\Omega \times [0,T] \times Z \to \erre$ such that
  \begin{align*}
    \norm[\big]{G(\omega,t,z,x)-G(\omega,t,z,y)} 
    &\leq g(\omega,t,z) \norm{x-y},\\
    \norm[\big]{G(\omega,t,z,x)}
    &\leq g(\omega,t,z) \bigl( 1 + \norm{x} \bigr)
  \end{align*}
  for all $\omega \in \Omega$, $t \in [0,T]$, $z\in Z$ and $x,\,y \in H$;
\item[(ii)] if $1< p<2$, then there exist functions $G_1$,
  $G_2: \Omega \times [0,T] \times Z \times H \to H$, satisfying the
  same measurability properties of $G$, with $G=G_1+G_2$, and
  $\mathscr{P} \otimes \mathscr{Z}$-measurable functions $g_1$,
  $g_2:\Omega \times [0,T] \times Z \to \erre$ such that, for
  $j \in \{1,2\}$,
  \begin{align*}
    \norm[\big]{G_j(\omega,t,z,x)-G_j(\omega,t,z,y)} 
    &\leq g_j(\omega,t,z) \norm{x-y},\\
    \norm[\big]{G_j(\omega,t,z,x)}
    &\leq g_j(\omega,t,z) \bigl( 1 + \norm{x} \bigr)
  \end{align*}
  for all $\omega \in \Omega$, $t \in [0,T]$, $z\in Z$ and $x,\,y \in H$.
  \end{itemize}
\end{itemize}
Further assumptions will be made when needed.

\medskip

The concept of solution to \eqref{eq:1} we shall work with is the
following.
\begin{definition}
  An $H$-valued adapted c\`adl\`ag process $u$ is a mild solution to
  \eqref{eq:1} if
  \begin{itemize}
  \item[(i)] $S(t-\cdot)f(u) \in L^1(0,t;H)$ for all $t \in [0,T]$
    $\P$-a.s.;
  \item[(ii)] $S(t-\cdot)B(u) \in L^2(0,t;\cL^2(K,H))$ for all
    $t \in [0,T]$ $\P$-a.s.;
  \item[(iii)] there exists $p>0$ such that
    $S(t-\cdot)G(u_-) \in \bG_p(0,t)$ for all $t \in [0,T]$;
  \item[(iv)] one has
    \[
      u = S(\cdot)u_0 + S \ast f(u) + S \diamond B(u) + S \diamond_\mu
      G(u_-)
    \]
    as an identity in the sense of modifications.
  \end{itemize}
\end{definition}

In order to formulate the well-posedness result in the mild sense for
\eqref{eq:1}, it is convenient to introduce an assumption depending on
a parameter $p \in \mathopen]0,\infty\mathclose[$:
\begin{itemize}
\item[(A5${}_p$)] Setting $g_1:=g_2:=g/2$ if
  $p \not\in \mathopen]1,2\mathclose[$, there exists a continuous increasing
  function $\kappa: \erre_+ \to \erre_+$, with $\kappa(0)=0$, such
  that
  \[
    1_{\{p>1\}} \biggl(
    \int_{Z \times [t_0,t_1]} g_1^p(\omega,s,z)\,d\nu
    \biggr)^{1/p}
    + \biggl( \int_{Z \times [t_0,t_1]} g_2^2(\omega,s,z)\,d\nu
    \biggr)^{1/2} \leq \kappa(t_1-t_0)
    \qquad \forall \omega \in \Omega.
  \]
\end{itemize}

\begin{thm}
  \label{thm:WP}
  Let $p>0$ and \emph{(A5${}_p$)} be satisfied. For any
  $u_0 \in L^p(\Omega;H)$, equation \eqref{eq:1} admits a unique mild
  solution $u \in \bS^p$ such that $\norm{u}_{\bS^p}\lesssim
  1+\norm{u_0}_{L^p(\Omega; H)}$, with implicit constant
  independent of $u_0$. Moreover, the solution map $u_0 \mapsto u$
  is Lipschitz continuous from $L^p(\Omega;H)$ to $\bS^p$.
\end{thm}
\begin{proof}
  We are going to use a fixed-point argument in the metric space
  $(\bS^p(0,T_0),d_{p,0,T_0})$, with $T_0$ sufficiently small. By a classical
  patching argument, this will imply existence and uniqueness of a
  solution in $\bS^p(0,T)$. Let $\Gamma$ be the map formally defined on
  $L^p(\Omega;H) \times \bS^p$ as
  \[
    \Gamma: (u_0,u) \longmapsto S(\cdot) u_0 + S \ast f(u) + S
    \diamond B(u) + S \diamond_\mu G(u_-).
  \]
  Let us show that $\Gamma$ is in fact well defined on
  $L^p(\Omega;H) \times \bS^p$ and that its image is contained in
  $\bS^p$: one has
  \begin{align}\label{est:u1}
    \norm[\big]{\Gamma(u_0,u)}_{\bS^p}
    &\lesssim \norm[\big]{S(\cdot)u_0}_{\bS^p}
      + \norm[\big]{S \ast f(u)}_{\bS^p}
      + \norm[\big]{S \diamond B(u)}_{\bS^p}
      + \norm[\big]{S \diamond_\mu G(u_-)}_{\bS^p},
  \end{align}
  where
  $\norm[\big]{S(\cdot)u_0}_{\bS^p} \leq \norm{u_0}_{L^p(\Omega;H)}$
  by contractivity of the semigroup $S$; the elementary
  lemma~\ref{lm:dc} and linear growth of $f$ imply
  \begin{align}\notag
    \norm[\big]{S \ast f(u)}_{\bS^p}
    &\leq \norm[\big]{f(u)}_{L^p(\Omega;L^1(0,T;H))}
      \leq C_f \bigl(1+\norm[\big]{u}_{L^p(\Omega;L^1(0,T;H))}\bigr)\\
    &\lesssim_p T C_f \bigl( 1 + \norm[\big]{u}_{\bS^p} \bigr);
    \label{est:u2}
  \end{align}
  similarly, proposition~\ref{prop:sc} yields
  \begin{align}\notag
    \norm[\big]{S \diamond B(u)}_{\bS^p}
    &\lesssim \norm[\big]{B(u)}_{L^p(\Omega;L^2(0,T;\cL^2(K,H)))}
      \leq C_B \bigl(1+\norm[\big]{u}_{L^p(\Omega;L^2(0,T;H))}\bigr)\\
    &\lesssim_p T^{1/2} C_B \bigl( 1 + \norm[\big]{u}_{\bS^p} \bigr);
    \label{est:u3}
  \end{align}
  finally, it follows by proposition~\ref{prop:scs} that
  $\norm[\big]{S \diamond_\mu G(u_-)}^p_{\bS^p} \lesssim
  \norm[\big]{G(u_-)}^p_{\bG^p}$, where, if
  $p \in \mathopen ]0,1] \cup [2,\infty[$,
  \begin{align}\notag
    \norm[\big]{G(u_-)}^p_{\bG^p}
    &= 1_{\{p>1\}} \E\int \norm{G(u_-)}^p\,d\nu
      + \E\biggl(\int \norm{G(u_-)}^2\,d\nu\biggr)^{p/2}\\ \notag
    &\leq 1_{\{p>1\}} \E\int g^p (1+\norm{u_-}^p)\,d\nu
      + \E\biggl(\int g^2 (1+\norm{u_-}^2)\,d\nu\biggr)^{p/2}\\
    &\lesssim \kappa^p(T) (1+\norm[\big]{u}^p_{\bS^p}),
    \label{est:u4}
  \end{align}
  and, similarly, if $p \in \mathopen]1,2\mathclose[$,
  \begin{align}\notag
    \norm[\big]{G(u_-)}^p_{\bG^p}&=
    \displaystyle \inf_{\tilde g_1+\tilde g_2=G(u_-)} \bigl( %
    \norm[\big]{\tilde g_2}^p_{L^p(\Omega;L^p(\nu;H))}+
    \norm[\big]{\tilde g_1}^p_{L^p(\Omega;L^2(\nu;H))} \bigr)\\ \notag
    &\leq \E\int \norm{G_1(u_-)}^p\,d\nu
      + \E\biggl(\int \norm{G_2(u_-)}^2\,d\nu\biggr)^{p/2}\\ \notag
    &\leq \E\int g_1^p (1+\norm{u_-}^p)\,d\nu
      + \E\biggl(\int g_2^2 (1+\norm{u_-}^2)\,d\nu\biggr)^{p/2}\\
    &\lesssim \kappa^p(T)(1+ \norm[\big]{u}^p_{\bS^p}).
    \label{est:u5}
  \end{align}
  Analogous arguments show that that $\Gamma(u_0,\cdot)$ is a
  contraction of $\bS^p(0,T_0)$, with $T_0$ to be chosen later. In
  fact, one has, with a slightly simplified notation,
  \begin{align*}
    \norm[\big]{\Gamma u - \Gamma v}^{1 \wedge p}_{\bS^p}
    &\leq
      \norm[\big]{S \ast (f(u)-f(v))}_{\bS^p}^{1 \wedge p}
      + \norm[\big]{S \diamond (B(u)-B(v))}_{\bS^p}^{1 \wedge p}\\
    &\quad + \norm[\big]{S \diamond_\mu (G(u_-)-G(v_-))}_{\bS^p}^{1 \wedge p}\\
    &=:A_1+A_2+A_3.
  \end{align*}
  Let us estimate the three terms separately. The Lipschitz continuity
  of $f$, $B$, and $G$ yields
  \begin{align*}
    \norm[\big]{S \ast (f(u)-f(v))}_{\bS^p(0,T_0)}
    &\leq \norm[\big]{f(u)-f(v)}_{L^p(\Omega;L^1(0,T_0;H))}\\
    &\leq T_0 C_f \norm[\big]{u-v}_{\bS^p(0,T_0)},\\
    \norm[\big]{S \diamond (B(u)-B(v))}_{\bS^p(0,T_0)}
    &\lesssim \norm[\big]{B(u)-B(v)}_{L^p(\Omega;L^2(0,T_0;\cL^2(K,H)))}\\
    &\leq T_0^{1/2} C_B \norm[\big]{u-v}_{\bS^p(0,T_0)},\\
    \norm[\big]{S \diamond_\mu (G(u_-)-G(v_-))}_{\bS^p(0,T_0)}
    &\lesssim \norm[\big]{G(u_-)-G(v_-)}_{\bG^p(0,T_0)}\\
    &\lesssim \kappa(T_0) \norm[\big]{u-v}_{\bS^p(0,T_0)},
  \end{align*}
  so that
  \begin{align*}
    A_1 &= \norm[\big]{S \ast (f(u)-f(v))}_{\bS^p(0,T_0)}^{1 \wedge p}
          \lesssim (T_0C_f)^{1 \wedge p} \norm[\big]{u-v}_{\bS^p(0,T_0)}^{1 \wedge p},\\
    A_2 &= \norm[\big]{S \diamond (B(u)-B(v))}_{\bS^p(0,T_0)}^{1 \wedge p}
          \lesssim \bigl( T_0^{1/2}C_B \bigr)^{1 \wedge p}
          \norm[\big]{u-v}_{\bS^p(0,T_0)}^{1 \wedge p},\\
    A_3 &= \norm[\big]{S \diamond (G(u_-)-G(v_-))}_{\bS^p(0,T_0)}^{1 \wedge p}
          \lesssim \kappa(T_0)^{1 \wedge p}
          \norm[\big]{u-v}_{\bS^p(0,T_0)}^{1 \wedge p}.
  \end{align*}
  Since $\kappa$ is continuous with $\kappa(0)=0$, it follows that
  there exists $T_0>0$ and a constant
  $\eta \in \mathopen ]0,1\mathclose[$, which depends on $T_0$, such that
  \[
    d_{p,0,T_0}(\Gamma u,\Gamma v) \leq \eta d_{p,0,T_0}(u,v),
  \]
  hence, by the Banach-Caccioppoli contraction principle, for any
  $u_0 \in L^p(\Omega;H)$ there exists a fixed point $u$ of the
  contraction $\Gamma(u_0,\cdot)$, which is thus the unique solution
  in $\bS^p(0,T_0)$ to \eqref{eq:1}. Choosing $T_0$ such that
  $T=nT_0$, with $n \in \mathbb{N}$, and repeating the same argument
  on each interval $[kT_0,(k+1)T_0]$, with $k \in \{1,\ldots,n-1\}$, a
  unique solution to \eqref{eq:1} can be constructed on the whole
  interval $[0,T]$. 
  Furthermore, for any $u_0\in L^p(\Omega; H)$, by \eqref{est:u1}-\eqref{est:u5},
  the unique solution $u=\Gamma(u_0,u)\in\bS^p(0,T)$ satisfies
  \[
  \norm{u}^p_{\bS^p(0,T)}\lesssim 1 + \norm{u_0}^p_{L^p(\Omega; H)}
  +(T+T^{1/2}+\kappa^p(T))\norm{u}^p_{\bS^p(0,T)},
  \]
  where the implicit constant is independent of $T$. Hence,
  there is $T_0\in(0,T)$ small enough such that 
  \[
  \norm{u}^p_{\bS^p(0,T_0)}\lesssim 1 + \norm{u_0}^p_{L^p(\Omega; H)}.
  \]
  Performing now a patching argument as above on
  $[0,T_0], \ldots, [(n-1)T_0, T]$ yields the desired estimate
  \[
  \norm{u}^p_{\bS^p(0,T)}\lesssim 1 + \norm{u_0}^p_{L^p(\Omega; H)}.
  \]
  The argument to show the Lipschitz-continuity of 
  $u_0\mapsto u$ is similar: let $u_{01}$, $u_{02} \in L^p(\Omega;H)$, 
  and $u_1$, $u_2 \in \bS^p(0,T)$ be the unique solutions to \eqref{eq:1} with
  initial datum $u_{01}$ and $u_{02}$, respectively.
  Using a patching argument as above, 
  it suffices to show that
  $u_0 \mapsto u$ is Lipschitz continuous on $[0,T_0]$. To this
  purpose, 
   One has
  \begin{align*}
    d_{p,0,T_0}(u_1,u_2)
    &= d_{p,0,T_0}\bigl(\Gamma(u_{01},u_1),\Gamma(u_{02},u_2)\bigr)\\
    &\leq d_{p,0,T_0}\bigl(\Gamma(u_{01},u_1),\Gamma(u_{02},u_1)\bigr)
      + d_{p,0,T_0}\bigl(\Gamma(u_{02},u_1),\Gamma(u_{02},u_2)\bigr)\\
    &\leq \norm[\big]{u_{01}-u_{02}}_{L^p(\Omega;H)}^{1 \wedge p}
      + \eta d_{p,0,T_0}(u_1,u_2),
  \end{align*}
  where $\eta<1$ is a positive constant (that depends on
  $T_0$). Rearranging terms and performing a patching
  argument as above immediately yields the Lipschitz
  continuity of $u_0 \mapsto u$.
\end{proof}

\begin{rmk}
  It immediately follows from the Lipschitz continuity of the solution
  map that one also has, in the same notation used above,
  \[
    \norm[\big]{u_1 - u_2}_{\bS^p} \lesssim
    \norm[\big]{u_{01}-u_{02}}_{L^p(\Omega;H)},
  \]
  with implicit constant depending on $T$ and $p$.
\end{rmk}


\ifbozza\newpage\else\fi
\section{G\^ateaux differentiability of the solution map}
\label{sec:G}
In the previous section we have shown that the solution map
$u_0 \mapsto u$ is Lipschitz continuous from $L^p(\Omega;H)$ to
$\bS^p$. We are now going to show that G\^ateaux differentiability of
the coefficients of \eqref{eq:1} implies G\^ateaux differentiability
of the solution map. For some applications (e.g. to study Kolmogorov
equations associated to stochastic PDEs) it is sufficient to consider
non-random initial data and to consider first-order derivatives as
linear maps from $H$ to $\bS^p$, i.e., roughly speaking, to consider
only non-random directions of differentiability. However, the more
general case of random initial data and random directions of
differentiability considered here as well as in the next sections is
conceptually not more difficult and, apart of being interesting in its
own right because treated at the natural level of generality, it is
necessary to study, for instance, higher-order stability issues of
stochastic models with respect to perturbations of the initial datum.

We shall make the following additional assumption, which is assumed to
hold throughout this section.
\begin{itemize}
\item[(G1)] The maps $f(\omega,t,\cdot)$ and $B(\omega,t,\cdot)$ are
  G\^ateaux differentiable for all
  $(\omega,t) \in \Omega \times [0,T]$, and the maps
  \[
    G(\omega,t,z,\cdot), \quad G_1(\omega,t,z,\cdot), \quad
    G_2(\omega,t,z,\cdot)
  \]
  are G\^ateaux differentiable for all
  $(\omega,t,z) \in \Omega \times [0,T] \times Z$.
\end{itemize}
The G\^ateaux derivatives of $f$, $B$ and $G$ (in their $H$-valued
argument) are denoted by
\begin{align*}
  D_{\mathcal G}f: \Omega \times [0,T] \times H &\longrightarrow \cL(H,H),\\
  D_{\mathcal G}B: \Omega \times [0,T] \times H &\longrightarrow \cL(H,\cL^2(K,H)),\\
  D_{\mathcal G}G: \Omega \times [0,T] \times Z \times H &\longrightarrow \cL(H,H).
\end{align*}
Recalling that $f$ and $B$ are Lipschitz continuous in their
$H$-valued argument, uniformly over $\Omega \times [0,T]$, we infer
that
\begin{align*}
  \norm[\big]{D_{\mathcal G}f(\omega,t,x_0)}_{\cL(H,H)}
  &\leq C_f,\\
  \norm[\big]{D_{\mathcal G}B(\omega,t,x_0)}_{\cL(H,\cL^2(K,H))}
  &\leq C_B
\end{align*}
for all $\omega \in \Omega$, $t \in [0,T]$, and $x_0 \in
H$. Similarly, the Lipschitz continuity of $G$ implies, if
$p \not\in \mathopen]1,2\mathclose[$, that
\[
\norm[\big]{D_{\mathcal G}G(\omega,t,z,x_0)}_{\cL(H,H)} \leq g(\omega,t,z),
\]
and, if $p \in \mathopen]1,2\mathclose[$, that
\begin{align*}
\norm[\big]{D_{\mathcal G}G(\omega,t,z,x_0)}_{\cL(H,H)} 
&\leq \norm[\big]{D_{\mathcal G}G_1(\omega,t,z,x_0)}_{\cL(H,H)} 
+ \norm[\big]{D_{\mathcal G}G_2(\omega,t,z,x_0)}_{\cL(H,H)}\\
&\leq g_1(\omega,t,z) + g_2(\omega,t,z),
\end{align*}
for all $\omega \in \Omega$, $t \in [0,T]$, $z \in Z$, and
$x_0 \in H$.

\medskip

We begin with two general results that will be extensively used in the
sequel.
The first lemma is an immediate corollary of the well-posedness
results.
\begin{lemma}
  \label{lm:y}
  Under the assumptions of Theorem~\ref{thm:WP}, let $u \in \bS^p$ be
  the unique mild solution to \eqref{eq:1} with initial condition
  $u_0 \in L^p(\Omega;H)$. For any $h \in L^p(\Omega;H)$, the linear
  stochastic evolution equation
  \begin{equation}
    \label{eq:y}  
    dy + Ay\,dt = D_{\mathcal G}f(u) y\,dt +D_{\mathcal G}B(u) y\,dW +
    \int_{Z}D_{\mathcal G} G(u_-)y_- \,d\bar{\mu}, \qquad y(0)=h,
  \end{equation}
  admits a unique mild solution $y \in \bS^p$ that depends
  continuously on the initial datum $h$.
\end{lemma}
\begin{proof}
  The linear maps $D_{\mathcal G}f(u)$ and $D_{\mathcal G}B(u)$ are bounded, uniformly over
  $\Omega \times [0,T]$, hence, a fortiori, Lipschitz
  continuous. Analogously, the linear map $D_{\mathcal G}G(u_-)$ has norm (and, a
  fortiori, Lipschitz constant) bounded by $g_1+g_2$ (with
  $g_1:=g_2:=g/2$ if $p \not\in \mathopen]1,2\mathclose[$) on
  $\Omega \times [0,T] \times Z$.  Theorem \ref{thm:WP} thus implies
  that, for any $h \in L^p(\Omega;H)$, \eqref{eq:y} admits a unique
  mild solution $y \in \bS^p$, which depends continuously on $h$.
\end{proof}

Note that, since the equation for $y$ is linear, it is immediate that
the map $h \mapsto y$ is linear and continuous from $L^p(\Omega;H)$ to
$\bS^p$.

\medskip

The next lemma will play a crucial role both in the proof of the
G\^ateaux differentiability of the solution map in this section, as
well as in the proof of its Fr\'echet differentiability in the next
section, taking into account Lemma~\ref{lm:Frech}.
\begin{lemma}
  \label{lm:fond}
  Under the assumptions of Theorem~\ref{thm:WP}, let
  $h \in L^p(\Omega;H)$ and $u$, $u_\varepsilon \in \bS^p$ the the
  unique mild solutions to \eqref{eq:1} with initial conditions $u_0$
  and $u_0+\varepsilon h$, respectively. Moreover, let $y \in \bS^p$
  be the unique mild solution to \eqref{eq:y} with initial condition
  $h$. One has
  \begin{align*}
    \norm[\big]{\varepsilon^{-1} (u_\varepsilon - u
    - \varepsilon y)}_{\bS^p}
    &\lesssim_p 
    \norm[\big]{\varepsilon^{-1}\bigl(
      f(u + \varepsilon y) - f(u)
      - \varepsilon D_{\mathcal G}f(u)y \bigr)}_{L^p(\Omega;L^1(0,T;H))}\\
    &\qquad + \norm[\big]{\varepsilon^{-1}
      \bigl(B(u + \varepsilon y) - B(u) 
      - \varepsilon D_{\mathcal G}B(u)y \bigr)}_{L^p(\Omega;L^2(0,T;\cL^2(K,H)))}\\
    &\qquad + \norm[\big]{ \varepsilon^{-1} \bigl(%
      G(u_- + \varepsilon y_-) - G(u_-)
      - \varepsilon D_{\mathcal G}G(u_-)y_- \bigr)}_{\bG^p}
  \end{align*}
\end{lemma}
\begin{proof}
  Let $[t_0,t_1] \subset [0,T]$, and consider the evolution equation
  \[
    dv + Av\,dt = f(v)\,dt + B(v)\,dW + G(v_-)\,d\bar\mu, \qquad
    v(t_0)=u(t_0).
  \]
  One easily sees that it admits a unique mild solution $v$, which
  coincides with the restriction of $u$ to $[t_0,t_1]$.  In particular,
  for any $t \geq t_0$,
  \begin{equation}
    \label{eq:storta}
    \begin{split}
      u(t) &= S(t-t_0)u(t_0) + \int_{t_0}^t S(t-s)f(u(s))\,ds
      + \int_{t_0}^t S(t-s)B(u(s))\,dW(s)\\
      &\quad + \int_{t_0}^t\!\!\int_Z S(t-s)G(z,u(s-))\,\bar\mu(dz,ds).
       \end{split}
     \end{equation}
  A completly analogous flow property holds for $u_\varepsilon$ and
  $y$. Then one has, by the triangle inequality,
  \begin{align*}
    \norm[\big]{\varepsilon^{-1}(u_\varepsilon - u - \varepsilon
    y)}_{\bS^p(t_0,t_1)} 
    &\lesssim_p \norm[\big]{\varepsilon^{-1}(u_\varepsilon(t_0) - u(t_0)%
      - \varepsilon y(t_0))}_{L^p(\Omega;H)}\\
    &\qquad + \norm[\big]{S \ast \varepsilon^{-1}\bigl(
      f(u_\varepsilon) - f(u)
      - \varepsilon D_{\mathcal G}f(u)y \bigr)}_{\bS^p(t_0,t_1)}\\
    &\qquad + \norm[\big]{S \diamond \varepsilon^{-1}
      \bigl(B(u_\varepsilon) - B(u) - \varepsilon D_{\mathcal G}B(u)y \bigr)}_{\bS^p(t_0,t_1)}\\
    &\qquad + \norm[\big]{S \diamond_\mu \varepsilon^{-1} \bigl(%
      G(u_{\varepsilon-}) - G(u_-)
      - \varepsilon D_{\mathcal G}G(u_-)y_- \bigr)}_{\bS^p(t_0,t_1)}\\
    &=: I_0 + I_1 + I_2 + I_3,
  \end{align*}
  where, by abuse of notation, the (deterministic and stochastic)
  convolutions are defined on $[t_0,t_1]$, in accordance to
  \eqref{eq:storta}, and $u_{\varepsilon-}:=(u_\varepsilon)_-$. We are
  going to estimate $I_1$, $I_2$ and $I_3$ separately. To simplify the
  notation, let us set, for a generic mapping $\phi$,
  \[
    [Q_{1,\varepsilon}\phi](u) := \frac{\phi(u_\varepsilon) - \phi(u +
      \varepsilon y)}{\varepsilon}, \qquad [Q_{2,\varepsilon}\phi](u)
    := \frac{\phi(u + \varepsilon y) - \phi(u)}{\varepsilon}
  \]
  (with obvious modifications if $u$ and $y$ are replaced by $u_-$ and $y_-$),
  and note that
  \[
    \frac{\phi(u_\varepsilon) - \phi(u)}{\varepsilon} =
    [Q_{1,\varepsilon}\phi](u) + [Q_{2,\varepsilon}\phi](u)
  \]
  (the formal operators $Q_{1,\varepsilon}$ and $Q_{2,\varepsilon}$
  clearly depend also on $y$, but we do not need to explicitly denote
  this fact).
  Recalling the elementary estimate of Lemma~\ref{lm:dc}, one has
  \begin{align*}
    I_1
    &\lesssim_p \norm[\big]{S \ast Q_{1,\varepsilon}f(u)}_{\bS^p(t_0,t_1)}
      + \norm[\big]{%
      S \ast \bigl(Q_{2,\varepsilon}f(u) - D_{\mathcal G}f(u)y \bigr)}_{\bS^p(t_0,t_1)}\\
    &\leq \norm[\big]{Q_{1,\varepsilon}f(u)}_{L^p(\Omega;L^1(t_0,t_1;H))}
      + \norm[\big]{Q_{2,\varepsilon}f(u) - D_{\mathcal G}f(u)y}_{L^p(\Omega;L^1(t_0,t_1;H))}\\
    &=: I_{11} + I_{12},
  \end{align*}
  where, by the Lipschitz continuity of $f$,
  \[
    I_{11} = \norm[\big]{Q_{1,\varepsilon}f(u)}_{L^p(\Omega;L^1(t_0,t_1;H))}
    \lesssim (t_1-t_0) \norm[\big]{%
      \varepsilon^{-1}(u_\varepsilon - u - \varepsilon y)}_{\bS^p(t_0,t_1)}.
  \]
  The terms $I_2$ and $I_3$ can be handled similarly, thanks to the
  maximal inequalities of {\S}\ref{ssec:maxi}:
  \begin{align*}
    I_2 
    &\lesssim_p \norm[\big]{S \diamond Q_{1,\varepsilon} B(u)}_{\bS^p(t_0,t_1)} +
      \norm[\big]{%
      S \diamond \bigl(Q_{2,\varepsilon}B(u) - D_{\mathcal G}B(u)y \bigr)}_{\bS^p(t_0,t_1)}\\
    &\lesssim_p \norm[\big]{Q_{1,\varepsilon} B(u)}_{L^p(\Omega;L^2(t_0,t_1;\cL^2(K,H)))}
      + \norm[\big]{Q_{2,\varepsilon}B(u) - D_{\mathcal G}B(u)y}_{L^p(\Omega;L^2(t_0,t_1;\cL^2(K,H)))}\\
    &=: I_{21} + I_{22},
  \end{align*}
  where
  \[
    I_{21} = \norm[\big]{%
      Q_{1,\varepsilon}B(u)}_{L^p(\Omega;L^2(t_0,t_1;\cL^2(K,H)))}
    \lesssim (t_1-t_0)^{1/2} \norm[\big]{%
      \varepsilon^{-1}(u_\varepsilon - u - \varepsilon y)}_{\bS^p(t_0,t_1)},
  \]
  and
  \begin{align*}
    I_3 
    &\lesssim_p \norm[\big]{S \diamond_\mu
      Q_{1,\varepsilon} G(u_-)}_{\bS^p(t_0,t_1)}
      + \norm[\big]{S \diamond_\mu
      \bigl(Q_{2,\varepsilon} G(u_-) - D_{\mathcal G}G(u_-)y_-\bigr)}_{\bS^p(t_0,t_1)}\\
    &\lesssim_p \norm[\big]{Q_{1,\varepsilon} G(u_-)}_{\bG^p(t_0,t_1)}
      + \norm[\big]{Q_{2,\varepsilon} G(u_-) - D_{\mathcal G}G(u_-)y_-}_{\bG^p(t_0,t_1)}\\
    &=: I_{31} + I_{32},
  \end{align*}
  where
  \[
    I_{31} \leq \kappa(t_1-t_0) \norm[\big]{\varepsilon^{-1}(u_\varepsilon
      - u - \varepsilon y)}_{\bS^p(t_0,t_1)}.
  \]
  Recalling that $\kappa$ is continuous with $\kappa(0)=0$,
  these estimates imply that 
  for every $\sigma>0$
  there exists $\delta>0$ such that, for
  any $t_0<t_1$ with $t_1-t_0 < \delta$, one has
  \begin{align*}
    \norm[\big]{%
    \varepsilon^{-1}(u_\varepsilon - u - \varepsilon
    y)}_{\bS^p(t_0,t_1)} 
    &\lesssim \sigma \norm[\big]{\varepsilon^{-1}(u_\varepsilon
      - u - \varepsilon y)}_{\bS^p(t_0,t_1)}\\
    &\quad+\norm[\big]{\varepsilon^{-1}(u_\varepsilon(t_0) - u(t_0)%
      - \varepsilon y(t_0))}_{L^p(\Omega;H)}\\
    &\quad + I_{12} + I_{22} + I_{32}.
  \end{align*}
  Fixing then $\sigma$ sufficiently small and rearranging the terms yields
  \begin{align*}
    \norm[\big]{%
    \varepsilon^{-1}(u_\varepsilon - u - \varepsilon
    y)}_{\bS^p(t_0,t_1)} 
    &\lesssim \norm[\big]{\varepsilon^{-1}(u_\varepsilon(t_0) - u(t_0)%
      - \varepsilon y(t_0))}_{L^p(\Omega;H)}\\
    &\quad + I_{12} + I_{22} + I_{32}.
  \end{align*}
  where the implicit constant depends on $\delta$ and $I_{12}$,
  $I_{22}$, $I_{32}$ are ``supported'' on $[t_0,t_1]$. Let
  $t_0=0<t_1<\cdots<t_{N-1}<t_N=T$ be a subdivision of the interval
  $[0,T]$ such that $t_{n}-t_{n-1}<\delta$ for all $n$. Then we have,
  for every $n \in \{1,\ldots,N\}$, with obvious meaning of the
  notation,
  \begin{align*}
    \norm[\big]{%
    \varepsilon^{-1}(u_\varepsilon - u - \varepsilon
    y)}_{\bS^p(t_{n-1},t_{n})} 
    &\lesssim \norm[\big]{\varepsilon^{-1}(u_\varepsilon(t_{n-1}) - u(t_{n-1})
      - \varepsilon y(t_{n-1}))}_{L^p(\Omega;H)}\\
    &\quad + [I_{12} + I_{22} + I_{32}](t_{n-1},t_n),
  \end{align*}
  where
  \begin{align*}
    &\norm[\big]{\varepsilon^{-1}(u_\varepsilon(t_{n-1}) - u(t_{n-1})
      - \varepsilon y(t_{n-1}))}_{L^p(\Omega;H)}\\
    &\hspace{3em} \leq \norm[\big]{%
    \varepsilon^{-1}(u_\varepsilon - u - \varepsilon
      y)}_{\bS^p(t_{n-2},t_{n-1})}\\
    &\hspace{3em} \lesssim
      \norm[\big]{\varepsilon^{-1}(u_\varepsilon(t_{n-2}) - u(t_{n-2})
      - \varepsilon y(t_{n-2}))}_{L^p(\Omega;H)}\\
    &\hspace{3em} \quad + [I_{12} + I_{22} + I_{32}](t_{n-2},t_{n-1}).
  \end{align*}
  Backward recursion thus yields
  \begin{align*}
    \norm[\big]{%
    \varepsilon^{-1}(u_\varepsilon - u - \varepsilon y)}_{\bS^p(0,T)}
    &\leq \sum_{n=1}^N \norm[\big]{
      \varepsilon^{-1}(u_\varepsilon - u
      - \varepsilon y)}_{\bS^p(t_{n-1},t_{n})}\\
    &\lesssim \norm[\big]{\varepsilon^{-1}(u_\varepsilon(0) - u(0)
      - \varepsilon y(0))}_{L^p(\Omega;H)}\\
    &\quad + \sum_{n=1}^N [I_{12} + I_{22} + I_{32}](t_{n-1},t_n)
  \end{align*}
  where the first summand on the right-hand side is zero.
  To conclude the proof it suffices to show that
  \[
    \sum_{n=1}^N I_{j2}(t_{n-1},t_n) \lesssim I_{j2}(0,T)
  \]
  for every $j \in \{1,2,3\}$. We shall show that this is true for
  $I_{32}$, as both other cases are entirely similar (in fact slightly
  simpler): it is enough to observe that, for any $\phi$ satisfying
  suitable measurability conditions and for any $q>0$, the
  obvious inequality
  \[
    \int_{Z \times [t_{n-1},t_n]} \abs{\phi}^q \,d\nu \leq
    \int_{Z \times [0,T]} \abs{\phi}^q \,d\nu
    \qquad \forall n \in \{1,\ldots,N\},
  \]
  implies $\norm{\phi}_{\bG^p(t_{n-1},t)} \leq \norm{\phi}_{\bG^p(0,T)}$, hence
  \[
    \sum_{n=1}^N I_{32}(t_{n-1},t_n) \leq N I_{32}(0,T).
    \qedhere
  \]
\end{proof}

The main result of this section is the following. Note that, since the
(standard) definition of G\^ateaux derivative requires a
Banach space framework, we shall confine ourself to the case
$p \in \mathopen[1,+\infty\mathclose[$.

\begin{thm}
  \label{thm:G1}
  Let $p\geq 1$ and \emph{(A5${}_p$)} be satisfied. Then
  the solution map of \eqref{eq:1} is G\^ateaux differentiable
  from $L^p(\Omega;H)$ to $\bS^p$, and its G\^ateaux derivative at
  $u_0$ is $(h \mapsto y) \in \cL(L^p(\Omega;H),\bS^p)$, where $y$ is
  the unique mild solution to \eqref{eq:y}.
\end{thm}
\begin{proof}
  By Lemma~\ref{lm:fond}, it is enough to show that
  \begin{equation}
    \label{eq:sepsi}
    \begin{split}
      &\norm[\big]{\varepsilon^{-1}\bigl( f(u + \varepsilon y) - f(u)
        - \varepsilon D_{\mathcal G}f(u)y \bigr)}_{L^p(\Omega;L^1(0,T;H))}\\
      &\quad + \norm[\big]{\varepsilon^{-1} \bigl(B(u + \varepsilon y)
        - B(u)
        - \varepsilon D_{\mathcal G}B(u)y \bigr)}_{L^p(\Omega;L^2(0,T;\cL^2(K,H)))}\\
      &\quad + \norm[\big]{ \varepsilon^{-1} \bigl(%
        G(u_- + \varepsilon y_-) - G(u_-) - \varepsilon D_{\mathcal G}G(u_-)y_-
        \bigr)}_{\bG^p}
    \end{split}
  \end{equation}
  converges to zero as $\varepsilon$ tends to zero.
  By assumption (G1) it immediately follows that, as
  $\varepsilon \to 0$,
  \begin{gather*}
  \norm[\big]{\varepsilon^{-1}\bigl( f(u + \varepsilon y) - f(u) 
  - \varepsilon D_{\mathcal G}f(u)y \bigr)}
  \longrightarrow 0,\\
  \norm[\big]{\varepsilon^{-1}
  \bigl(B(u + \varepsilon y) - B(u) 
  - \varepsilon D_{\mathcal G}B(u)y \bigr)}_{\cL^2(K,H)}
  \longrightarrow 0
  \end{gather*}
  for a.a. $(\omega,t) \in \Omega \times [0,T]$. Moreover, recalling
  that the operator norms of $D_{\mathcal G}f$ and $D_{\mathcal G}B$ are bounded by the
  Lipschitz constants of $f$ and $B$, respectively, the triangle
  inequality yields
  \begin{align*}
    &\norm[\big]{\varepsilon^{-1}\bigl( f(u + \varepsilon y) - f(u) 
      - \varepsilon D_{\mathcal G}f(u)y \bigr)}\\
    &\hspace{3em} + \norm[\big]{\varepsilon^{-1}
    \bigl(B(u + \varepsilon y) - B(u) 
      - \varepsilon D_{\mathcal G}B(u)y \bigr)}_{\cL^2(K,H)}
      \lesssim \bigl( C_f + C_B \bigr) \norm{y}
  \end{align*}
  for a.a. $(\omega,t)$. Since $y \in \bS^p$, the right-hand side
  belongs to $L^p(\Omega;L^1(0,T))$ as well as to
  $L^p(\Omega;L^2(0,T))$, hence the first two terms in
  \eqref{eq:sepsi} converge to zero as $\varepsilon \to 0$ by the
  dominated convergence theorem.
  Similarly, setting $G_1:=G_2:=G/2$ if $p \geq 2$, one has
  \begin{align*}
    &\norm[\big]{ \varepsilon^{-1} \bigl(%
      G(u_- + \varepsilon y_-) - G(u_-) - \varepsilon D_{\mathcal G}G(u_-)y_-
      \bigr)}_{\bG^p}\\
    &\hspace{3em}\lesssim_p \norm[\big]{ \varepsilon^{-1} \bigl(%
      G_1(u_- + \varepsilon y_-) - G_1(u_-) - \varepsilon D_{\mathcal G}G_1(u_-)y_-
      \bigr)}_{L^p(\Omega;L^2(\nu;H))}\\
    &\hspace{3em}\quad + \norm[\big]{ \varepsilon^{-1} \bigl(%
      G_2(u_- + \varepsilon y_-) - G_2(u_-) - \varepsilon D_{\mathcal G}G_2(u_-)y_-
      \bigr)}_{L^p(\Omega;L^p(\nu;H))},
  \end{align*}
  where the implicit constant is equal to $1$ for
  $p \in [1,2\mathclose[$, and to $2$ for $p \geq 2$.
  Since
  \[
    \norm[\big]{ \varepsilon^{-1} \bigl(%
      G_j(u_- + \varepsilon y_-) - G_j(u_-) - \varepsilon D_{\mathcal G}G_j(u_-)y_-
      \bigr)} \longrightarrow 0
  \]
  as $\varepsilon \to 0$, as well as 
  \[
    \norm[\big]{ \varepsilon^{-1} \bigl(%
      G_j(u_- + \varepsilon y_-) - G_j(u_-) - \varepsilon D_{\mathcal G}G_j(u_-)y_-
      \bigr)} \leq g_j \norm{y}
  \]
  for all $(t,z) \in [0,T] \times Z$, $\P$-almost surely, for both
  $j=1$ and $j=2$, one has, thanks to (A5${}_p$) and the dominated
  convergence theorem, recalling that $y \in \bS^p$,
  \begin{gather*}
    \norm[\big]{ \varepsilon^{-1} \bigl(%
      G_1(u_- + \varepsilon y_-) - G_1(u_-) - \varepsilon D_{\mathcal G}G_1(u_-)y_-
      \bigr)}_{L^2(\nu;H)} \longrightarrow 0,\\
    \norm[\big]{ \varepsilon^{-1} \bigl(%
      G_2(u_- + \varepsilon y_-) - G_2(u_-) - \varepsilon D_{\mathcal G}G_2(u_-)y_-
      \bigr)}_{L^p(\nu;H)} \longrightarrow 0
  \end{gather*}
  $\P$-a.s. as $\varepsilon \to 0$. A further application of the
  dominated convergence theorem hence yields that the third term in
  \eqref{eq:sepsi} converges to zero as $\varepsilon \to 0$, thus
  completing the proof.
\end{proof}


\ifbozza\newpage\else\fi
\section{Fr\'echet differentiability of the solution map}
\label{sec:F}
We are going to show that the Fr\'echet differentiability of the
coefficients of \eqref{eq:1} implies the Fr\'echet differentiability
of the solution map. We shall work under the following assumption,
that is assumed to hold throughout this section.
\begin{itemize}
\item[(F)] The maps $f(\omega,t,\cdot)$ and $B(\omega,t,\cdot)$ are
  Fr\'echet differentiable for all
  $(\omega,t) \in \Omega \times [0,T]$, and the maps
  \[
    G(\omega,t,z,\cdot), \quad G_1(\omega,t,z,\cdot), \quad
    G_2(\omega,t,z,\cdot)
  \]
  are Fr\'echet differentiable for all
  $(\omega,t,z) \in \Omega \times [0,T] \times Z$.
\end{itemize}
The Fr\'echet derivatives of $f$ and $B$ (in their $H$-valued
argument), denoted by
\begin{align*}
  Df: \Omega \times [0,T] \times H &\longrightarrow \cL(H,H),\\
  DB: \Omega \times [0,T] \times H &\longrightarrow \cL(H,\cL^2(K,H)),
\end{align*}
satisfy the boundedness properties
\begin{align*}
  \norm[\big]{Df(\omega,t,x)}_{\cL(H,H)} &\leq C_f,\\
  \norm[\big]{DB(\omega,t,x)}_{\cL(H,\cL^2(K,H))} &\leq C_B
\end{align*}
for all $(\omega,t,x) \in \Omega \times [0,T] \times H$ (see \S~\ref{ssec:der}).
Similarly, and in complete analogy to the previous section, the
Lipschitz continuity assumptions on $G$, $G_1$ and $G_2$ imply that,
\begin{align*}
  \norm[\big]{DG(\omega,t,z,x)}_{\cL(H,H)} &\leq g(\omega,t,z), & &p \geq 2,\\
  \norm[\big]{DG_j(\omega,t,z,x)}_{\cL(H,H)} &\leq g_j(\omega,t,z),
  & &p \in [1,2\mathclose[, \quad j=1,2.
\end{align*}

The main result of this section is the following theorem, which states
that the solution map is Fr\'echet differentiable along subspaces of
vectors with finite higher moments.
\begin{thm}
  \label{thm:F1}
  Let $q > p \geq 1$. If \emph{(A5${}_p$)} and \emph{(A5${}_q$)} hold, then the
  solution map of \eqref{eq:1} is Fr\'echet differentiable from
  $L^p(\Omega;H)$ to $\bS^p$ along $L^q(\Omega;H)$ and its Fr\'echet
  derivative at $u_0\in L^p(\Omega;H)$ is the map
  $h \mapsto y \in \cL(L^q(\Omega;H),\bS^p)$, where $y$ is the unique
  mild solution to the stochastic evolution equation
  \begin{equation}
    \label{eq:y_F}  
    dy + Ay\,dt = Df(u)y\,dt + DB(u)y\,dW +
    \int_Z DG(u_-)y_- \,d\bar{\mu}, \qquad y(0)=h.
  \end{equation}
\end{thm}
\begin{proof}
  For any $h \in L^q(\Omega; H)$, equation \eqref{eq:y_F} admits a
  unique mild solution $y \in \bS^q$, as it follows immediately by the
  boundedness properties of the Fr\'echet derivatives of $f$, $B$ and
  $G$, and by hypothesis (A5${}_q$).  Therefore the map $h \mapsto y$
  is well defined from $L^q(\Omega; H)$ to $\bS^q$, and it is
  obviously linear and continuous. To prove that this map is the Fr\'echet
  derivative of the solution map $u_0 \mapsto u$, thanks to the
  characterization of Fr\'echet differentiability of
  Lemma~\ref{lm:Frech}, it is enough to show that
  \[
    \lim_{\varepsilon \to 0} \, \norm[\big]{%
      \varepsilon^{-1}(u_\varepsilon - u - \varepsilon y)}_{\bS^p} = 0
  \]
  uniformly over $h$ belonging to bounded subsets of $L^q(\Omega;H)$.
  By Lemma~\ref{lm:fond}, for this it suffices to show that each term
  in \eqref{eq:sepsi} converges to zero uniformly with respect to $h$
  belonging to the unit ball of $L^q(\Omega;H)$.  Since
  $h \mapsto y \in \cL(L^q(\Omega;H),\bS^q)$, it is
  evident that if $h$ belongs to $B_1(L^q(\Omega; H))$
  then $y(h)$ belongs to $B_R(\bS^q)$,
  where $R:=\norm{h \mapsto y}_{\cL(L^q(\Omega; H), \bS^q)}$.
  Hence, denoting by $I_j$, $j=1,2,3$, the terms appearing in
  \eqref{eq:sepsi}, by homogeneity
  \[
    \sup_{h \in B_1(L^q(\Omega; H))} (I_1+I_2+I_3)(y(h)) \leq
    \sup_{y \in B_R(\bS^q)} (I_1+I_2+I_3)(y).
  \]
  Hence it
  suffices to show that $I_1$, $I_2$ and $I_3$ converge to zero
  uniformly with respect to $y$ bounded in $\bS^q$. That is, we need
  to show that, for any $R>0$ and $\vartheta>0$, there exists
  $\varepsilon_0=\varepsilon_0(R,\vartheta)$ such that
  $\abs{\varepsilon}<\varepsilon_0$ implies $I_j(y) < \vartheta$ for
  all $y \in B_R(\bS^q)$ and $j \in \{1,2,3\}$.
  For any measurable $E \subset \Omega$, one clearly has
  \begin{align*}
    I_1 &= \norm[\big]{\varepsilon^{-1}\bigl(f(u+\varepsilon y) - f(u)%
          - \varepsilon Df(u)y\bigr)}_{L^p(\Omega;L^1(0,T;H))}\\
        &\leq \norm[\big]{\varepsilon^{-1}\bigl(f(u+\varepsilon y) - f(u)%
          - \varepsilon Df(u)y\bigr)}_{L^p(E;L^1(0,T;H))}\\
        &\quad + \norm[\big]{\varepsilon^{-1}\bigl(f(u+\varepsilon y) - f(u)%
          - \varepsilon Df(u)y\bigr)}_{L^p(E^c;L^1(0,T;H))}\\
        &=: I_1(E) + I_1(E^c),
  \end{align*}
  where, by the Lipschitz continuity of $f$,
  \begin{align*}
    I_1(E)^p
    &= \E 1_E \biggl( \int_0^T \norm[\big]{%
      (f(u+\varepsilon y)-f(u))/\varepsilon - Df(u)y} \,dt \biggr)^p\\
    &\leq \bigl(2 T C_f\bigr)^p \E 1_E \, (y^*_T)^p.
  \end{align*}
  The set $Y := \{ (y^*_T)^p:\, y \in B_R(\bS^q) \}$ is bounded in
  $L^{q/p}(\Omega)$, with $q>p$, hence uniformly integrable on
  $(\Omega,\cF,\P)$. In particular, for any $\vartheta > 0$ there exists
  $\sigma>0$ such that, for any $E \in \cF$ with $\P(E)<\sigma$, one
  has
  \[
    \E 1_E \, (y^*_T)^p < \biggl( \frac{\vartheta}{2 C_f T}
    \biggr)^p, \qquad \forall y \in B_R(\bS^q),
  \]
  hence $I_1(E) \leq \vartheta$.
  Let $y \in B_R(\bS^q)$ be arbitrary but fixed. Markov's inequality
  yields, for any $n > 0$,
  \[
    \P\bigl(y^*_T > n\bigr) \leq \frac{\E (y^*_T)^q}{n^q}
    \leq \frac{R^q}{n^q}.
  \]
  Therefore there exists $n>0$ such that, setting $E:=\{y^*_T>n\}$,
  one has $I_1(E) < \vartheta$. It is important to note that $n$
  depends on $R$, but not on $y$, while $E$ depends on $y$.
  The Fr\'echet differentiability hypothesis on $f$ amounts to saying
  that, for any $x \in H$ and $n \in \mathbb{N}$,
  \[
    \lim_{\varepsilon \to 0} \sup_{z \in B_n(H)}
    \norm[\big]{(f(x+\varepsilon z) - f(x))/\varepsilon - Df(x)z} = 0.
  \]
  In particular, one has
  \[
    \lim_{\varepsilon \to 0} \sup_{z \in B_n(H)}
    \norm[\big]{(f(u(\omega,t) + \varepsilon z) -
      f(u(\omega,t)))/\varepsilon - Df(u(\omega,t))z} = 0
  \]
  for a.a. $(\omega,t) \in E^c \times [0,T]$, where, by the Lipschitz
  continuity of $f$,
  \[
    \sup_{z \in B_n(H)} \norm[\big]{(f(u(\omega,t) + \varepsilon z) -
      f(u(\omega,t)))/\varepsilon - Df(u(\omega,t))z} \lesssim 2n C_f,
  \]
  for a.a. $(\omega,t) \in E^c \times [0,T]$. Therefore, by the
  dominated convergence theorem,
  \[
    \lim_{\varepsilon \to 0} \norm[\bigg]{%
      \sup_{z \in B_n(H)} \norm[\big]{(f(u(\omega,t) + \varepsilon z)
        - f(u(\omega,t)))/\varepsilon -
        Df(u(\omega,t))z}}_{L^p(E^c;L^1(0,T))} =0,
  \]
  that is, for any $\vartheta > 0$ there exists $\varepsilon_1$
  depending only on $\vartheta$ and $n$ such that
  \[
    \norm[\bigg]{%
      \sup_{z \in B_n(H)} \norm[\big]{(f(u + \varepsilon z) -
        f(u))/\varepsilon - Df(u)z}}_{L^p(E^c;L^1(0,T))} < \vartheta
  \]
  for all $\varepsilon$ such that
  $\abs{\varepsilon}<\varepsilon_1(\vartheta,n)$. It remains to
  observe that
  \begin{align*}
    & \norm[\big]{(f(u(\omega,t) + \varepsilon y(\omega,t)) -
      f(u(\omega,t)))/\varepsilon - Df(u(\omega,t))y(\omega,t)}\\
    &\hspace{3em} \leq
      \sup_{z \in B_n(H)} \norm[\big]{(f(u(\omega,t) + \varepsilon z) -
      f(u(\omega,t)))/\varepsilon - Df(u(\omega,t))z}
  \end{align*}
  for a.a. $(\omega,t) \in E^c \times [0,T]$ to get that
  $I_1(E^c) < \vartheta$ for all $\varepsilon$ such that
  $\abs{\varepsilon}<\varepsilon_1(\vartheta,n)$. Since $n$ depends
  only on $R$, we conclude that there exists
  $\varepsilon_1=\varepsilon_1(\vartheta,R)$ such that
  $I_1 < 2\vartheta$ for all
  $\abs{\varepsilon}<\varepsilon_1(\vartheta,R)$.
  
  Let us now consider the term $I_2$: the argument is similar to the
  one just carried out, so we provide slightly less detail.  We have
  to show that $I_2$ converges to $0$ uniformly with respect to
  $y \in B_R(\bS^q)$.  For any measurable $E \subset \Omega$, one has,
  with obvious meaning of the notation,
  \[
    I_2 \leq I_2(E) + I_2(E^c),
  \]
  where, by the Lipschitz-continuity of $B$,
  \[
    I_2(E)^p \leq (2 C_B)^p T^{p/2} \E 1_E\,(y^*_T)^p.
  \]
  Choosing $E$ as before, using the uniform integrability of the
  family $Y$ combined with the Markov inequality, we infer that for
  any $\vartheta > 0$ there exists $n>0$ such that
  $I_2(E) < \vartheta$.  The Fr\'echet differentiability of $B$
  implies that, for any $x \in H$,
  \[
    \lim_{\varepsilon \to 0} \sup_{z \in B_n(H)}
    \norm[\big]{(B(x+\varepsilon z) - B(x))/\varepsilon -
      DB(x)z}_{\cL^2(K,H)} = 0
  \]
  in $E^c \times [0,T]$, where, by the Lipschitz continuity of $B$,
  \[
    \sup_{z \in B_n(H)} \norm[\big]{(B(u(\omega,t) + \varepsilon z) -
      B(u(\omega,t)))/\varepsilon - DB(u(\omega,t))z}_{\cL^2(K,H)}
    \lesssim 2n C_B,
  \]
  for a.a. $(\omega,t) \in E^c \times [0,T]$. Hence, the dominated
  convergence theorem yields
  \[
    \lim_{\varepsilon \to 0} \norm[\bigg]{%
      \sup_{z \in B_n(H)} \norm[\big]{(B(u(\omega,t) + \varepsilon z)
        - B(u(\omega,t)))/\varepsilon -
        DB(u(\omega,t))z}_{\cL^2(K,H)}}_{L^p(E^c;L^2(0,T))} =0,
  \]
  that is, for any $\vartheta > 0$ there exists $\varepsilon_2$
  depending only on $\vartheta$ and $n$ such that
  \[
    \norm[\bigg]{%
      \sup_{z \in B_n(H)} \norm[\big]{(B(u + \varepsilon z) -
        B(u))/\varepsilon - DB(u)z}_{\cL^2(K,H)}}_{L^p(E^c;L^2(0,T))}
    < \vartheta
  \]
  for all $\varepsilon$ such that
  $\abs{\varepsilon}<\varepsilon_2(\vartheta,n)$, from which also
  $I_2(E^c) < \vartheta$ for all $\varepsilon$ such that
  $\abs{\varepsilon}<\varepsilon_2(\vartheta,n)$. Hence, there exists
  $\varepsilon_2=\varepsilon_2(\vartheta, R)$ such that
  $I_2<2\vartheta$ for all $\varepsilon$ with
  $\abs{\varepsilon} < \varepsilon_2(\vartheta, R)$.
 
  The convergence to zero of $I_3$ as $\varepsilon \to 0$, uniformly
  with respect to $y \in B_R(\bS^q)$, while still similar to the above
  arguments, is slightly more delicate as random measures are
  involved.  As already shown in the proof of Theorem~\ref{thm:G1},
  one has, recalling that Fr\'echet differentiability implies
  G\^ateaux differentiability,
  \[
    \norm[\big]{\varepsilon^{-1}\bigl(G(u_-+\varepsilon y_-)
      - G(u_-)\bigr) - DG(u_-)y_-}_{\bG^p} \longrightarrow 0
  \]
  as $\varepsilon \to 0$. We need to show that the convergence holds
  uniformly over $y$ bounded in $\bS^q$. Let $R>0$ and
  $y \in B_R(\bS^q)$. For any measurable $E \in \cF$, the Lipschitz
  continuity assumptions on $G$ and (A5${}_p$) imply, setting
  $G_1:=G_2:=G$ if $p \geq 2$, that
  \begin{align*}
    &\norm[\big]{1_E \bigl(\varepsilon^{-1}\bigl(G(v+\varepsilon
      w)-G(v)\bigr) - DG(v)w\bigr)}^p_{\bG^p}\\
    &\hspace{3em} \leq \E 1_E \int \norm[\big]{
      \varepsilon^{-1}\bigl(G_1(u_-+\varepsilon y_-) - G_1(u_-)\bigr)
      - DG_1(u_-)y_-}^p\,d\nu\\
    &\hspace{3em} \quad + \E 1_E \biggl( \int \norm[\big]{
      \varepsilon^{-1}\bigl(G_2(u_-+\varepsilon y_-) -
      G_2(u_-)\bigr) - DG_2(u_-)y_-}^2\,d\nu\biggr)^{p/2}\\
    &\hspace{3em} \lesssim_p \kappa(T)^p \E 1_E (y^*_T)^p.
  \end{align*}
  As the set $\{(y^*_T)^p:\, y \in B_R(\bS^q)\}$ is bounded in
  $L^{q/p}(\Omega)$, hence uniformly integrable, for any $\vartheta>0$
  there exists $n>0$ (by Markov's inequality) such that, choosing
  $E:=\{y^*>n\}$ as before, we have
  \[ 
  \kappa(T)^p \E 1_E (y^*_T)^p < \vartheta.
  \]
  On $E^c$ one has, possibly outside a set of $\P$-measure zero, for
  both $j=1$ and $j=2$,
  \begin{align*}
    &\norm[\big]{\varepsilon^{-1}\bigl(G_j(u_-+\varepsilon y_-)
      - G_j(u_-) \bigr) - DG_j(u_-)y_-}\\ 
    &\hspace{3em} \leq \sup_{\xi \in B_n(H)}
    \norm[\big]{\varepsilon^{-1}\bigl(G_j(u_- + \varepsilon\xi) - G_j(u_-)\bigr)
      - DG_j(u_-)\xi}
  \end{align*}
  where the right-hand side converges to zero by the characterization
  of Fr\'echet differentiability of Lemma~\ref{lm:Frech}, and is
  bounded by $2n g_j$ for all $(t,z) \in [0,T] \times Z$. Since
  $g_1 \in L^p(\nu)$ and $g_2 \in L^2(\nu)$ $\P$-a.s. in $E^c$, the
  dominated convergence theorem and (A5${}_p$) yield
  \[
  \norm[\big]{1_{E^c} \bigl(\varepsilon^{-1}\bigl(G(u_-+\varepsilon
    y_-)-G(u_-)\bigr) - DG(u_-)y_-\bigr)}_{\bG^p} \longrightarrow 0
  \]
  as $\varepsilon \to 0$, uniformly with respect to
  $y \in B_R(\bS^q)$. Proceeding exactly as in the case of $I_1$, we
  conclude that there exists
  $\varepsilon_3=\varepsilon_3(\vartheta,R)$ such that
  $I_3<2\vartheta$ for all $\abs{\varepsilon}<\varepsilon_3$.
  
  We have thus shown that
  $\varepsilon^{-1}(u_\varepsilon-u-\varepsilon y) \to 0$ in
  $\bS^p(0,T)$, uniformly over $h$ in any bounded subset of
  $L^q(\Omega;H)$, as claimed.
\end{proof}

\ifbozza\newpage\else\fi
\section{Fr\'echet differentiability of higher order}
\label{sec:F+}
In this section we show that the $n$-th order Fr\'echet
differentiability of the coefficients of \eqref{eq:1}, in a suitable
sense, implies the $n$-th order Fr\'echet differentiability of the
solution map. We shall work under the following assumptions, that are
stated in terms of the parameter $n \in \mathbb{N}$, $n \geq 2$:
\begin{itemize}
\item[(F${}^n$)] The maps $f(\omega,t,\cdot)$ and $B(\omega,t,\cdot)$
  are $n$ times Fr\'echet differentiable for all
  $(\omega,t) \in \Omega \times [0,T]$, and the maps
  $G(\omega,t,z,\cdot)$, $G_i(\omega,t,z,\cdot)$, $i=1,2$, are $n$
  times Fr\'echet differentiable for all
  $(\omega,t,z) \in \Omega \times [0,T] \times Z$. Moreover, there exists a
  constant $m \geq 0$ such that, for every $j=2,\ldots,n$,
  \[
    \norm[\big]{D^jf(\omega,t,x)}_{\cL_j(H;H)}
    + \norm[\big]{D^jB(\omega,t,x)}_{\cL_j(H;\cL^2(K,H))}
    \lesssim 1 + \norm{x}^m
  \]
  for all $(\omega,t,x) \in \Omega \times [0,T] \times H$, and
  \begin{align*}
    \norm[\big]{D^jG(\omega,t,z,x)}_{\cL_j(H;H)}
    &\lesssim g(\omega,t,z) \bigl( 1 + \norm{x}^m \bigr),\\
    \norm[\big]{D^jG_i(\omega,t,z,x)}_{\cL_j(H;H)}
    &\lesssim g_i(\omega,t,z) \bigl( 1 + \norm{x}^m \bigr), \qquad i=1,2.
  \end{align*}
\end{itemize}
We also stipulate that (F${}^1$) is simply hypothesis (F) of the
previous section. It would be possible to replace the functions $g$,
$g_1$ and $g_2$ with different ones, thus reaching a bit more
generality, but it does not seem to be worth the (mostly notational)
effort.

\begin{example}\label{ex:Fn}
  Let us give an explicit example where assumption (F${}^n$)
  is satisfied with a suitable choice of $m>0$ and not for $m=0$.
  We shall consider $B=G=0$ for simplicity and
  concentrate only on $f$: typical examples for $B$ and $G$
  can be produced following the same argument.
  Let $H=L^2(D)$, where $D\subset\erre^d$ is 
  a smooth bounded domain, and consider the function
  \[
  \gamma:\erre\to\erre\,, \qquad\gamma(r):=\int_0^r\sin(s^2)\,ds\,, \quad r\in\erre\,.
  \]
  It is not difficult to check that $\gamma\in C^\infty(\erre)$, $\gamma$ is 
  Lipschitz-continuous (hence $\gamma'\in C_b(\erre)$), and 
  \[
  |\gamma^{(j)}(r)|\lesssim 1+ |r|^{j-1} \qquad\forall\,r\in\erre\,,\quad\forall\,j\in\mathbb{N},\quad j\geq1\,.
  \]
  However, the derivatives $\gamma^{(j)}$ are {\em not} bounded in $\erre$
  for any $j\geq2$.
  Furthermore, let us fix $\mathcal L\in \cL(H, L^\infty(D))$, and define the 
  operator
  \[
  f: H\to H\,, \qquad
  f(u):=\gamma(\mathcal L u)\,, \quad u\in H\,.
  \]
  Clearly, $f$ is well-defined, Lipschitz-continuous and linearly bounded,
  so that (A2) is satisfied. 
  Moreover, using the fact that $\mathcal L\in \cL(H,L^\infty(D))$
  it a standard matter to check that $f$ is Fr\'echet-differentiable, and
  its derivative is given by
  \[
  Df: H \to \cL(H,H)\,, \qquad
  Df(u)h=\gamma'(\mathcal Lu)\mathcal Lh\,, \quad u,h\in H\,,
  \]
  so that also assumption (F) is satisfied. Note in particular that the first 
  derivative $Df$ is also bounded in $H$ thanks to the Lipschitz-continuity 
  of $f$.
  Furthermore, using the fact that $\mathcal L\in \cL(H,L^\infty(D))$
  a direct computation shows that for every $j\in\mathbb{N}$,
  with $j>1$, $f$ is Fr\'echet-differentiable $j$-times and 
  \[
  D^jf(u):H\to\cL_j(H^j;H)\,, \qquad
  D^jf(u)(h_1,\ldots,h_j)=\gamma^{(j)}(\mathcal L u)\mathcal Lh_1\ldots\mathcal Lh_j\,, 
  \quad u,h_1,\ldots,h_j\in H\,.
  \]
  For every $j>1$,
  by the H\"older inequality and the properties of $\gamma$ and $\mathcal L$
  we have that 
  \[
  \norm{D^jf(u)}_{\cL_j(H^j;H)}\lesssim 1+ \norm{\gamma^{(j)}(\mathcal Lu)}_{L^\infty(D)}
  \lesssim_\gamma 1+ \norm{\mathcal Lu}_{L^\infty(D)}^{j-1}\lesssim_{\mathcal L} 1+ \norm{u}_H^{j-1}\,,
  \]
  so that assumption (F${}^n$) is satisfied for every $n$ with the choice $m=n-1$.
  However, note that the higher-order derivatives of $f$ are {\em not} bounded in $H$
  because of the choice of the function $\gamma$: hence,
  coefficients $f$ in this form cannot be treated using
  available results in literature (as for example \cite{cm:JFA10}).
  On the other hand, these are nonetheless included in our analysis.
\end{example}

\medskip

In the following we shall write, for compactness of notation, $\bL^q$
in place of $L^q(\Omega;H)$. If $u$ (identified with the solution map
$u_0 \mapsto u:\bL^p \to \bS^p$, which is well defined if assumption
(A5${}_p$) holds) is $n$ times Fr\'echet differentiable along
$\bL^{q_1},\ldots,\bL^{q_n}$, we have
\[
  D^nu(u_0) \in \cL_n\bigl(\bL^{q_1},\ldots,\bL^{q_n};\bS^p\bigr).
\]
Under the assumptions of Theorem~\ref{thm:F1}, $u$ is once Fr\'echet
differentiable and $v:=Du(u_0)$ satisfies the equation
\[
  v = S(\cdot)I + S * Df(u)v + S \diamond DB(u)v
  + S \diamond_\mu DG(u_-)v_-,
\]
where $I$ is the identity map. This equation has to be interpreted in
the sense that, for any $h \in \bL^q$, $q>p$, setting $y:=[Du(u_0)]h$,
one has
\[
  y = S(\cdot)h + S*Df(u)y + S \diamond DB(u)y + S
  \diamond_\mu DG(u_-)y_-.
\]
Note that by Lemma~\ref{lm:y} this equation admits a unique solution $y \in \bS^p$ also for
$h \in \bL^p$, and that $h \mapsto y \in \cL(\bL^p,\bS^p)$. However,
if $h$ belongs only to $\bL^p$, we can no longer claim that
$h \mapsto y$ is the Fr\'echet derivative of $u_0 \mapsto u$, as
Theorem~\ref{thm:F1} does not necessarily apply.

We are now going to introduce a system of equations, indexed by
$n \geq 2$, that are formally expected to be satisfied by $D^ju(u_0)$,
$j=1,\ldots,n$, if they exist. 
For any $n\geq 2$, the equation for $u^{(n)}$
can be written as
\begin{equation}
  \label{eq:xx}
\begin{split}
  d\diff{n}{u} + A \diff{n}{u}\,dt
  &= \bigl( \Psi_n + Df(u)\diff{n}{u} \bigr)\,dt
    + \bigl( \Phi_n + DB(u)\diff{n}{u} \bigr)\,dW\\
  &\qquad + \int_Z\bigl( \Theta_n + DG(u_-)\diff{n}{u}_-
  \bigr)\,d\bar\mu, \qquad \diff{n}{u}(0) = 0,
\end{split}
\end{equation}
where $\Psi_n$, $\Phi_n$ and $\Theta_n$ are the formal $n$-th
Fr\'echet derivatives of $f(u)$, $B(u)$ and $G(u_-)$, respectively,
excluding the terms involving the (formal) derivative of $u$ of order
$n$. More precisely, 
assume that $E_1$, $E_2$ and $E_3$ are Banach
spaces and $\phi:E_1 \to E_2$, $F:E_2 \to E_3$ are $n$ times Fr\'echet
differentiable. The chain rule implies that there exists a function
$\tilde{\Phi}^F_n$ such that
\[
  D^n[F(\phi)] = \tilde{\Phi}^F_n\bigl( D\phi,\ldots,D^{n-1}\phi \bigr)
  + DF(\phi)D^n\phi.
\]
We set
$\Phi_n :=
\tilde{\Phi}^B_n\bigl(\diff{1}{u},\diff{2}{u},\ldots,\diff{n-1}{u}\bigr)$. The
definition of $\Psi_n$ and $\Theta_n$ is, \emph{mutatis mutandis},
identical.

The concept of solution for equation \eqref{eq:xx}
is intended as in the case of the first order derivative equation,
i.e.~in the sense of testing against arbitrary directions.
More precisely, we shall say that
\[
  \diff{n}{u} \in \cL_n\bigl(\bL^{q_1},\ldots,\bL^{q_n};\bS^p\bigr),
  \qquad p,q_1,\ldots,q_n \geq 1,
\]
is a solution to \eqref{eq:xx} if, for any
\[
  (h_1,\ldots,h_n) \in \bL^{q_1} \times \cdots \times \bL^{q_n},
\]
the process $\diff{n}{u}(h_1,\ldots,h_n) \in \bS^p$ satisfies
\begin{align*}
  \diff{n}{u}(h_1,\ldots,h_n)
  &= S * \Psi_n(h_1,\ldots,h_n)
    + S * Df(u) \diff{n}{u}(h_1,\ldots,h_n)\\
  &\quad +S \diamond \Phi_n(h_1,\ldots,h_n)
    + S \diamond DB(u) \diff{n}{u}(h_1,\ldots,h_n)\\
  &\quad +S \diamond_\mu \Theta_n(h_1,\ldots,h_n)
    + S \diamond_\mu DG(u_-) \diff{n}{u}_-(h_1,\ldots,h_n).
\end{align*}

Let us show some properties of the coefficients
$\Psi_n$, $\Phi_n$ and $\Theta_n$.
We are going to use some algebraic properties of the
``representing'' map $\tilde{\Phi}^F_n$. In particular, although a (kind
of) explicit expression for $\tilde{\Phi}_n^F$ can be written in terms
of a variant of the Fa\`a di Bruno formula (as it was done
for example in \cite{Jentzen}), for our purposes it
suffices to know that $\tilde{\Phi}_n^F$ is a sum of terms of the form
\[
  D^jF(\phi)\bigl(D^{\alpha_1}\phi,\ldots,D^{\alpha_j}\phi\bigr),
\]
with $j \in \{2,\ldots,n\}$, $\alpha_1 + \cdots + \alpha_j = n$,
$\alpha_i \geq 1$ for all $i \in \{1,\ldots,j\}$. Moreover, since
$D^n[F(\phi)]$ is an $n$-linear map on $E_1^n$ with values in $E_3$
(with $E_1^n$ being the cartesian product of $E_1$ by itself $n$-times),
one has that, for any $(h_1,\ldots,h_n) \in E_1^n$,
$D^n[F(\phi)](h_1,\ldots,h_n)$ is a sum of terms of the form
\[
  D^jF(\phi)\bigl(%
           D^{\alpha_1}\phi(h_{\sigma(1)},\ldots,h_{\sigma(\alpha_1)}),
           \ldots,
           D^{\alpha_j}\phi(h_{\sigma(A_j+1)},\ldots,h_{\sigma(n)})
           \bigr),
\]
where $A_j:=\alpha_1+\cdots+\alpha_{j-1}$, and $\sigma$ is an element of
the permutation group of $\{1,\ldots,n\}$.
We shall also need the following identities, that we write already in
the specific form needed later, although they are obviously a
consequence of the definition of $\tilde{\Phi}^F_n$:
\begin{equation}
  \label{eq:sai}
  \begin{split}
  D\Psi_n &= \Psi_{n+1} - D^2f(u)(u',\diff{n}{u}),\\
  D\Phi_n &= \Phi_{n+1} - D^2B(u)(u',\diff{n}{u}),\\
  D\Theta_n &= \Theta_{n+1} - D^2G(u_-)(u'_-,\diff{n}{u}_-),
  \end{split}
\end{equation}
where we have written, as customary, $u'$ in place of $\diff{1}{u}$.
We are going to write, for the convenience of the reader, the first
three formal derivatives of $B(u)$ and the expressions for $\Phi_n$
(the corresponding calculations for $f(u)$, $G(u_-)$, $\Psi_n$, and
$\Theta_n$ are entirely analogous). One has
\begin{align*}
  D[B(u)] &= DB(u)u',\\
  D^2[B(u)] &= D^2B(u)\bigl(u', u'\bigr)+ DB(u)u^{(2)},\\
  D^3[B(u)] &= D^3B(u)\bigl(u', u', u'\bigr) +
              D^2B(u)\bigl(u^{(2)}, u'\bigr) + D^2B(u)\bigl(u', u^{(2)}\bigr)\\
          &\quad  + D^2B(u)\bigl(u', u^{(2)}\bigr) +
            DB(u)u^{(3)}\\
          &= D^3B(u)\bigl(u', u', u'\bigr) %
            + 3 D^2B(u)\bigl(u^{(2)}, u'\bigr) + DB(u)u^{(3)},\\
  \Phi_1&=0,\\
  \Phi_2&= D^2B(u)\bigl(u', u'\bigr),\\
  \Phi_3&=D^3B(u)\bigl(u', u', u'\bigr) +
          3D^2B(u)\bigl(u^{(2)}, u'\bigr),
\end{align*}
where we have used Schwarz's theorem on the symmetry of higher-order
continuous Fr\'echet derivatives.

The first result that we present concerns the existence and uniqueness
of solutions to equation \eqref{eq:xx} in the sense specified above.
More precisely, we show in the next proposition that equation
\eqref{eq:xx} admits a unique solution $u^{(n)}$, belonging to
$\cL_n\bigl(\bL^{p_1},\ldots,\bL^{p_n};\bS^p\bigr)$. Note that to
study differentiability we shall restrict to the case $p_1=\cdots=p_n$
(see Remark~\ref{coeff_p} below). However, since well-posedness for
linear stochastic equations for multilinear maps such as \eqref{eq:xx}
could be interesting in its own right, we shall provide a general
result considering arbitrary $p_1,\ldots,p_n$.

\begin{prop}
  \label{prop:aux}
  Let $n \geq 1$ and $p,p_0,p_1,\ldots,p_n \geq 1$ be such that
  $u_0 \in \bL^p \cap \bL^{mp_0} = \bL^{p \vee mp_0}$ and
  \begin{equation}
    \label{eq:recu}
    \frac{n-1}{p_0} + \frac{1}{p_1} + \cdots + \frac{1}{p_n}
    \leq \frac{1}{p}.
  \end{equation}
  Assume that
  \begin{itemize}
  \item[\emph{(i)}] hypothesis \emph{(F${}^n$)} is satsfied;
  \item[\emph{(ii)}] hypothesis \emph{(A5${}_r$)} holds for all
    $r \in [p,\max_{i \geq 1} p_i]\cup\{mp_0\}$.
  \end{itemize}
  Then \eqref{eq:xx} admits a unique solution
  \[
    \diff{n}{u} \in \cL_n\bigl(\bL^{p_1},\ldots,\bL^{p_n};\bS^p\bigr).
  \]
\end{prop}
\begin{proof}
  First of all, let us explain why $\diff{n}{u}$, if it exists, must be $n$-linear (in
  the algebraic sense). Since $u'=Du$ is indeed a linear map, we can use
  induction as follows: assuming that $\diff{j}{u}$ is $j$-linear for
  all $j<k$, with $k\in\{2,\ldots,n\}$, we are going to show that $\diff{k}{u}$ is $k$-linear. The
  inductive assumption and the functional form of $\Psi_k$, $\Phi_k$,
  and $\Theta_k$ imply that they are $k$-linear. Considering the equation
  \begin{align*}
  v &= S * \Psi_k(h_1,\ldots,h_k)
    + S * Df(u) v\\
  &\quad +S \diamond \Phi_k(h_1,\ldots,h_k)
    + S \diamond DB(u) v\\
  &\quad +S \diamond_\mu \Theta_k(h_1,\ldots,h_k)
    + S \diamond_\mu DG(u_-) v_-,
  \end{align*}
  assuming that a solution exists for every
  $(h_1,\ldots,h_k) \in \bL^{q_1} \times \cdots \times \bL^{q_k}$,
  $q_1,\ldots,q_k \geq 1$, it suffices to show that the map
  $(h_1,\ldots,h_k) \mapsto v$ is $k$-linear, which is immediate.\\
  Let us focus now on existence.
  We are going to reason by induction on the order of (formal)
  derivation $k\in\{1,\ldots,n\}$. The claim is certainly true for $k=1$:
  Theorem~\ref{thm:G1} implies, thanks for assumption (ii), that
  $u' \in \cL(\bL^r,\bS^r)$ for every $r \in [p,\max_{i \geq 1} p_i]$,
  hence also $u' \in \cL(\bL^s,\bS^r)$
  for every $s \geq r$, as then $\bL^s$ is contractively
  embedded in $\bL^r$.
  Let us now assume the the claim is true for all $j\leq k\in\{1,\ldots,n-1\}$, and
  consider $h_j\in\bL^{p_j}$ with $p_j\geq p$, for $j=1,\ldots,k+1$,
  such that
  \[
    \frac{k}{p_0} + \frac1{p_1} + \ldots + \frac1{p_{k+1}} \leq
    \frac1p.
  \]
  In order to control the $\bS^p$ norm of
  $\diff{k+1}{u}(h_1,\ldots,h_{k+1})$ it is enough to estimate
  \begin{gather*}
    \norm[\big]{\Psi_{k+1}(h_1,\ldots,h_{k+1})}_{L^p(\Omega;L^1(0,T;H)},\\
    \norm[\big]{\Phi_{k+1}(h_1,\ldots,h_{k+1})}_{L^p(\Omega;L^2(0,T;\cL^2(K,H)))},\\
    \norm[\big]{\Theta_{k+1}(h_1,\ldots,h_{k+1})}_{\bG^p}.
  \end{gather*}
  In fact, recalling that $Df(u)$,
  $DB(u)$ and $DG(u)$ are bounded linear operators (in the same sense
  as in the proofs of Theorems~\ref{thm:G1} and \ref{thm:F1}), one
  has, for any $[t_0,t_1] \subseteq [0,T]$, omitting the indication of
  the arguments $(h_j)$ for simplicity of notation,
  \begin{align*}
    \norm[\big]{\diff{k+1}{u}}_{\bS^p(t_0,t_1)}
    &\leq \norm[\big]{\diff{k+1}{u}(t_0)}_{\bL^p}
      + \norm[\big]{S \ast \Psi_{k+1}}_{\bS^p(t_0,t_1)}
      + \norm[\big]{S \ast Df(u) \diff{k+1}{u}}_{\bS^p(t_0,t_1)}\\
    &\quad + \norm[\big]{S \diamond \Phi_{k+1}}_{\bS^p(t_0,t_1)}
      + \norm[\big]{S \diamond DB(u) \diff{k+1}{u}}_{\bS^p(t_0,t_1)}\\
    &\quad + \norm[\big]{S \diamond_\mu \Theta_{k+1}}_{\bS^p(t_0,t_1)}
      + \norm[\big]{S \diamond_\mu DG(u) \diff{k+1}{u}}_{\bS^p(t_0,t_1)}\\
    &\lesssim \norm[\big]{\Psi_{k+1}}_{L^p(\Omega;L^1(t_0,t_1;H))}
      + \norm[\big]{\Phi_{k+1}}_{L^p(\Omega;L^2(t_0,t_1;\cL^2(K,H)))}
      + \norm[\big]{\Theta_{k+1}}_{\bG^p(t_0,t_1)}\\
    &\quad + \norm[\big]{\diff{k+1}{u}(t_0)}_{\bL^p}
      + \kappa(t_1-t_0) \norm[\big]{\diff{k+1}{u}}_{\bS^p(t_0,t_1)},
  \end{align*}
  where the implicit constant does not depend on $t_1-t_0$ (and also
  not on $k$). We proceed now as in the proof of 
  Lemma~\ref{lm:fond}: choosing $T_0>0$ sufficiently small and partitioning
  $[0,T]$ in intervals of lenght not exceeding $T_0$, it follows from
  $\diff{k+1}{u}(0)=0$ that
  \begin{equation}
    \label{eq:ecce}
    \norm[\big]{\diff{k+1}{u}}_{\bS^p} \lesssim
    \norm[\big]{\Psi_{k+1}}_{L^p(\Omega;L^1(0,T;H))}
    + \norm[\big]{\Phi_{k+1}}_{L^p(\Omega;L^2(0,T;\cL^2(K,H)))}
    + \norm[\big]{\Theta_{k+1}}_{\bG^p},
  \end{equation}
  as claimed. Let us consider the second term on the right-hand side
  of the previous inequality (the first one can be handled in a
  completely similar way). As already seen, the generic term in
  $\Phi_{k+1}(h_1,\ldots,h_{k+1})$ is of the form
  \[
  D^jB(u)\bigl(%
           \diff{\alpha_1}{u}(h_{\sigma(1)},\ldots,h_{\sigma(\alpha_1)}),
           \ldots,
           \diff{\alpha_j}{u}(h_{\sigma(\beta+1)},\ldots,h_{\sigma(k)})
           \bigr),
  \]
  where $j \in \{2,\ldots,k+1\}$, $\alpha_1+\cdots+\alpha_j=k+1$,
  $\beta:=\alpha_1+\cdots+\alpha_{j-1}$, and $\sigma$ is an element of
  the permutation group of $\{1,\ldots,k+1\}$. Since $j \geq 2$
  implies
  \[
  1 + (\alpha_1-1) + \cdots + (\alpha_j-1)
  = \alpha_1 + \cdots + \alpha_j + 1 - j = k+2-j \leq k,
  \]
  one has
  \begin{align*}
  \frac{1}{p_0}
  &+ \frac{\alpha_1-1}{p_0} + \frac{1}{p_{\sigma(1)}}
    + \cdots + \frac{1}{p_{\sigma(\alpha_1)}}\\
  &+ \frac{\alpha_2-1}{p_0} + \frac{1}{p_{\sigma(\alpha_1+1)}}
    + \cdots + \frac{1}{p_{\sigma(\alpha_1+\alpha_2)}}\\
  &\vdotswithin{+}\\
  &+ \frac{\alpha_j-1}{p_0}
    + \frac{1}{p_{\sigma(\alpha_1+\cdots+\alpha_{j-1}+1)}}
    + \cdots + \frac{1}{p_{\sigma(k+1)}}\\
  &\hspace{3em} \leq \frac{n}{p_0} + \frac{1}{p_1} + \cdots
    + \frac{1}{p_{k+1}} = \frac{1}{p},
  \end{align*}
  so that setting
  \begin{align*}
    \frac{1}{\tilde{p}_1} &:=  \frac{\alpha_1-1}{p_0} + \frac{1}{p_{\sigma(1)}}
    + \cdots + \frac{1}{p_{\sigma(\alpha_1)}},\\
    \frac{1}{\tilde{p}_2} &:= \frac{\alpha_2-1}{p_0} + \frac{1}{p_{\sigma(\alpha_1+1)}}
    + \cdots + \frac{1}{p_{\sigma(\alpha_1+\alpha_2)}},\\
  &\vdotswithin{:=}\\
  \frac{1}{\tilde{p}_j} &:= \frac{\alpha_j-1}{p_0}
    + \frac{1}{p_{\sigma(\alpha_1+\cdots+\alpha_{j-1}+1)}}
    + \cdots + \frac{1}{p_{\sigma(k+1)}},
  \end{align*}
  it holds
  \[
    \frac{1}{p_0} + \frac{1}{\tilde{p}_1} + \cdots +
    \frac{1}{\tilde{p}_{j}} \leq \frac{1}{p}.
  \]
  Assumption (F${}^n$) now implies
  \begin{align*}
    &\norm[\big]{D^jB(u)\bigl(%
           \diff{\alpha_1}{u}(h_{\sigma(1)},\ldots,h_{\sigma(\alpha_1)}),
           \ldots,
           \diff{\alpha_j}{u}(h_{\sigma(\beta+1)},\ldots,h_{\sigma(k)})
           \bigr)}_{\cL^2(K,H)}\\
    &\hspace{3em} \lesssim \bigl( 1 + \norm{u}^m \bigr)
      \norm[\big]{%
      \diff{\alpha_1}{u}(h_{\sigma(1)},\ldots,h_{\sigma(\alpha_1)})}_{\cL^2(K,H)}
      \; \cdots\\
    &\hspace{3em} \qquad \cdots \;
      \norm[\big]{%
      \diff{\alpha_j}{u}(h_{\sigma(\beta+1)},\ldots,h_{\sigma(k)})}_{\cL^2(K,H)},
  \end{align*}
  which yields, thanks to the estimate
  $\norm{\cdot}_{L^2(0,T)} \leq T^{1/2} \norm{\cdot}_{L^\infty(0,T)}$,
  \begin{align*}
    &\norm[\big]{D^jB(u)\bigl(%
      \diff{\alpha_1}{u}(h_{\sigma(1)},\ldots,h_{\sigma(\alpha_1)}),
      \ldots,
      \diff{\alpha_j}{u}(h_{\sigma(\beta+1)},\ldots,h_{\sigma(k)})
      \bigr)}_{L^2(0,T;\cL^2(K,H))}\\
    &\hspace{3em} \lesssim \bigl( 1 + u^{*m} \bigr)
      \bigl(\diff{\alpha_1}{u}(h_{\sigma(1)},\ldots,h_{\sigma(\alpha_1)})\bigr)^*
      \; \cdots \; \bigl(%
      \diff{\alpha_j}{u}(h_{\sigma(\beta+1)},\ldots,h_{\sigma(k)})\bigr)^*,
  \end{align*}
  where the implicit constant depends also on $T$. 
  Here and in the following we write, for simplicity of notation,
  $\phi^*:=\phi^*_T$ for any c\`adl\`ag function $\phi$.
  H\"older's inequality yields
  \begin{align*}
    &\norm[\big]{D^jB(u)\bigl(%
      \diff{\alpha_1}{u}(h_{\sigma(1)},\ldots,h_{\sigma(\alpha_1)}),
      \ldots,
      \diff{\alpha_j}{u}(h_{\sigma(\beta+1)},\ldots,h_{\sigma(k)})
      \bigr)}_{L^p(\Omega;L^2(0,T;\cL^2(K,H)))}\\
    &\hspace{3em} \lesssim \bigl( 1 + \norm[\big]{u^{*m}}_{\bL^{p_0}} \bigr)
      \norm[\big]{%
      \diff{\alpha_1}{u}(h_{\sigma(1)},\ldots,h_{\sigma(\alpha_1)})}%
      _{\bS^{\tilde{p}_1}}
      \; \cdots \; \norm[\big]{%
      \diff{\alpha_j}{u}(h_{\sigma(\beta+1)},\ldots,h_{\sigma(k+1)})}%
    _{\bS^{\tilde{p}_{j}}},
  \end{align*}
  where, as before, $\beta:=\alpha_1+\cdots+\alpha_{j-1}$.
  It follows by the definition of $\tilde{p}_1,\ldots,\tilde{p}_j$ and
  the inductive assumption that
  \begin{align*}
    \diff{\alpha_1}{u}
    &\in \cL_{\alpha_1}\bigl(\bL^{p_{\sigma(1)}},\ldots,\bL^{p_{\sigma(\alpha_1)}};
      \bS^{\tilde{p}_1}\bigr),\\
    &\vdotswithin{\in}\\
    \diff{\alpha_j}{u}
    &\in \cL_{\alpha_j}\bigl(\bL^{p_{\sigma(\beta+1)}},\ldots,\bL^{p_{\sigma(k+1)}};
      \bS^{\tilde{p}_j}\bigr),
  \end{align*}
  hence, recalling that
  $\norm[\big]{u^{*m}}_{\bL^{p_0}} = \norm[\big]{u}^m_{\bS^{mp_0}}
  \lesssim 1+\norm[\big]{u_0}^m_{\bL^{mp_0}}$ by Theorem~\ref{thm:WP},
  \begin{align*}
    &\norm[\big]{D^jB(u)\bigl(%
      \diff{\alpha_1}{u}(h_{\sigma(1)},\ldots,h_{\sigma(\alpha_1)}),
      \ldots,
      \diff{\alpha_j}{u}(h_{\sigma(\beta+1)},\ldots,h_{\sigma(k)})
      \bigr)}_{L^p(\Omega;L^2(0,T;\cL^2(K,H)))}\\
    &\hspace{3em} \lesssim \bigl( 1 + \norm[\big]{u_0}^m_{\bL^{mp_0}} \bigr)
      \norm[\big]{h_{\sigma(1)}}_{\bL^{p_{\sigma(1)}}} \; \cdots \;
      \norm[\big]{h_{\sigma(n+1)}}_{\bL^{p_{\sigma(k+1)}}}\\
    &\hspace{3em} \lesssim \bigl( 1 + \norm[\big]{u_0}^m_{\bL^{mp_0}} \bigr)
      \norm[\big]{h_1}_{\bL^{p_1}} \; \cdots \;
      \norm[\big]{h_{k+1}}_{\bL^{p_{k+1}}}.
  \end{align*}
  Estimating the $\bG^p$ norm of $\Theta_{k+1}$ is similar: using the
  same notation used thus far, the generic term in
  $\Theta_{k+1}(h_1,\ldots,h_{k+1})$ is of the type
  \[
    D^jG(u_-)\bigl(%
    \diff{\alpha_1}{u}_-(h_{\sigma(1)},\ldots,h_{\sigma(\alpha_1)}),
    \ldots,
    \diff{\alpha_j}{u}_-(h_{\sigma(\beta+1)},\ldots,h_{\sigma(k)})
    \bigr),
  \]
  and hypothesis (F${}^n$) implies
  \begin{align*}
    &\norm[\big]{D^jG_i(u_-)\bigl(%
    \diff{\alpha_1}{u}_-(h_{\sigma(1)},\ldots,h_{\sigma(\alpha_1)}),
    \ldots,
    \diff{\alpha_j}{u}_-(h_{\sigma(\beta+1)},\ldots,h_{\sigma(k)})
      \bigr)}\\
    &\hspace{3em} \lesssim g_i \bigl( 1 + \norm{u}^m \bigr)
      \norm[\big]{%
      \diff{\alpha_1}{u}(h_{\sigma(1)},\ldots,h_{\sigma(\alpha_1)})}
      \; \cdots\\
    &\hspace{3em} \qquad \cdots \;
      \norm[\big]{%
      \diff{\alpha_j}{u}(h_{\sigma(\beta+1)},\ldots,h_{\sigma(k)})}
  \end{align*}
  for all $(t,z) \in [0,T] \times Z$, $\P$-a.s., for both $i=1$ and
  $i=2$ (we can identify again $g_1$ and $g_2$ with $g$ depending on
  the value of $p$, and similarly for $G_1$ and $G_2$). This yields,
  after standard computations already detailed more than once,
  \begin{align*}
    &\norm[\big]{D^jG(u_-)\bigl(%
    \diff{\alpha_1}{u}_-(h_{\sigma(1)},\ldots,h_{\sigma(\alpha_1)}),
    \ldots,
    \diff{\alpha_j}{u}_-(h_{\sigma(\beta+1)},\ldots,h_{\sigma(k)})
      \bigr)}_{\bG^p}\\
    &\hspace{3em} \lesssim \kappa(T) \bigl( 1 + \norm[\big]{u^{*m}}_{\bL^{p_0}}\bigr)
      \norm[\big]{\diff{\alpha_1}{u}%
      (h_{\sigma(1)},\ldots,h_{\sigma(\alpha_1)})}_{\bS^{\tilde{p}_1}}
      \; \cdots\\
    &\hspace{3em} \qquad \cdots \;
      \norm[\big]{\diff{\alpha_j}{u}%
      (h_{\sigma(\beta+1)},\ldots,h_{\sigma(k)})}_{\bS^{\tilde{p_j}}},
  \end{align*}
  It hence follows by the inductive assumption, as
  before, that
  \begin{align*}
    &\norm[\big]{D^jG(u)\bigl(%
      \diff{\alpha_1}{u}(h_{\sigma(1)},\ldots,h_{\sigma(\alpha_1)}),
      \ldots,
      \diff{\alpha_j}{u}(h_{\sigma(\beta+1)},\ldots,h_{\sigma(k)})
      \bigr)}_{\bG^p}\\
    &\hspace{3em} \lesssim \bigl( 1 + \norm[\big]{u_0}^m_{\bL^{mp_0}}
      \bigr) \norm[\big]{h_1}_{\bL^{p_1}} \; \cdots \;
      \norm[\big]{h_{k+1}}_{\bL^{p_{k+1}}}.
  \end{align*}
  Since $p_1,\ldots,p_{k+1}$ were arbitrary, we have proved that
  $k/p_0 + \sum_{j=1}^{k+1} 1/p_j \leq 1/p$ implies
  \[
    \diff{k+1}{u} \in \cL_{k+1}(\bL^{p_1},\ldots,\bL^{p_{k+1}};\bS^p),
  \]
  thus completing the induction argument by arbitrariness of $k$.
\end{proof}  

\begin{rmk}\label{coeff_p}
  If $p_1=\cdots=p_n=q$, condition \eqref{eq:recu} becomes
  \[
  \frac{n-1}{p_0} + \frac{n}{q} \leq \frac{1}{p},
  \]
  which implies $q \geq np$ and $p_0 \geq (n-1)p$, hence $p_0 \geq p$
  if $n \geq 2$. In particular, if $q=np$, then $p_0=+\infty$,
  i.e. $u_0$ must be bounded almost surely. If $q>np$, then $p_0$ will
  also be finite, and strictly larger than $p$ if $n \geq 2$.
  Furthermore, if $q > (n+nm-m)p$, then $u_0 \in \bL^q$ implies
  $\diff{n}{u} \in \cL_n(\bL^q;\bS^p)$. In fact, for this to be true
  it suffices that $\bL^q \subseteq \bL^{mp_0 \vee p}$, which is
  equivalent to $q \geq mp_0 \vee p$. But since $q \geq np \geq p$, we
  can simply choose $q=mp_0$, which yields, excluding the case
  $p_0=+\infty$,
  \[
  \frac{(n-1)m}{q} + \frac{n}{q} < \frac{1}{p},
  \]
  or, equivalently, $q > (n+nm-m)p$.

  We repeat, however, that even under these conditions we cannot yet
  claim that $\diff{n}{u}$ identifies the $n$-th Fr\'echet derivative
  of $u$. In fact, we shall prove that $D^nu$ satisfies the equation
  for $\diff{n}{u}$ when ``tested'' on $(\bL^q)^n$, with $q$ satisfying a
  strictly stronger constraint than just $q>(n+mn-m)p$.
\end{rmk}

Before considering Fr\'echet differentiability of $n$-th order, we
need some preparations. The following two lemmata are used to apply
the theorem on the Fr\'echet differentiability of the composition of
two Fr\'echet differentiable functions.

By the assumptions (A2), (A3) and (A4), it follows immediately
that the superposition operators associated to 
$f$, $B$ and $G$ on $\bS^p$, i.e. 
$\phi\mapsto f(\phi), B(\phi), G(\phi_-)$,
can be considered as maps, denoted by the
same symbols for simplicity,
\begin{align*}
    f: \bS^p &\longrightarrow \bS^p,\\
    B: \bS^p &\longrightarrow L^p(\Omega;L^2(0,T;\cL^2(K,H))),\\
    G: \bS^p &\longrightarrow \bG^p.
  \end{align*}

\begin{lemma}
  \label{lm:Fcomp}
  Let $p \geq 1$, $r>0$, $q \geq 1$, and $n \in \mathbb{N}$ satisfy
  \[
    \frac1r + \frac{n}{q} < \frac1p, \qquad
    \frac{n+m}{q} < \frac{1}{p}.
  \]
  If hypothesis \emph{(F${}^n$)} is
  satisfied, then
  $f$, $B$ and $G$
  are $n$-times Fr\'echet differentiable in
  $\bS^{mr} \cap \bS^p = \bS^{mr \vee p}$ along $\bS^q$, with
  \begin{align*}
    D^jf: \bS^{mr \vee p}
    &\longrightarrow \cL_j(\bS^q;\bS^p),\\
    D^jB: \bS^{mr \vee p}
    &\longrightarrow
    \cL_j\bigl(\bS^q;L^p(\Omega;L^2(0,T;\cL^2(K,H)))\bigr),\\
    D^jG: \bS^{mr \vee p}
    &\longrightarrow \cL_j(\bS^q;\bG^p)
  \end{align*}
  for all $j \in \{1,\ldots,n\}$.
\end{lemma}
\begin{proof}
  We proceed by induction on $j$, and we treat only the third term, as
  all other cases are analogous (in fact slightly simpler). If $j=1$,
  the proof is exactly the same as the corresponding one of
  Theorem~\ref{thm:F1}. In particular, one has
  \[
    DG(\omega,t,z,\cdot): H \to \cL(H,H) \qquad \forall (\omega,t,z) \in
    \Omega \times [0,T] \times Z,
  \]
  hence, given $v \in \bS^p$ and $w \in \bS^q$ with $q>1 \cdot p=p$,
  \[
    \norm[\big]{\varepsilon^{-1}\bigl(G(v_-+\varepsilon w_-)-G(v_-)\bigr)
      - DG(v_-)w_-\bigr)}_{\bG^p} \longrightarrow 0
  \]
  as $\varepsilon \to 0$, uniformly over $w$ belonging to bounded
  subsets of $\bS^q$.
  Assuming now that the statement is true for
  $j \in \{1,\ldots,n-1\}$, let us show that it also holds for
  $j+1$. By the inductive hypothesis we thus have
  \[
    D^jG: \bS^{mr} \cap \bS^p \to \cL_j(\bS^q;\bG^p).
  \]
  Let $u \in \bS^p$ and $v_1,\ldots,v_{j+1} \in \bS^q$. The
  $(j+1)$-th Fr\'echet derivatives
  \[
    D^{j+1}G(\omega,t,z,\cdot), \;
    D^{j+1}G_i(\omega,t,z,\cdot): H \to \cL_{j+1}(H;H), \qquad i=1,2,
  \]
  exists for all $(\omega,t,z) \in \Omega \times [0,T] \times Z$,
  hence, setting $\mathbf{v}_k:=(v_1,\ldots,v_k)$, $k=1,\ldots,n$, one
  has, as $\varepsilon \to 0$,
  \begin{equation}
    \label{eq:kk}
    \norm[\Big]{\frac{1}{\varepsilon} %
      \bigl(D^jG(u_-+\varepsilon (v_{j+1})_-)\mathbf{v}_{j-} %
      - D^jG(u_-)\mathbf{v}_{j-} \bigr) %
      - D^{j+1}G(u_-)(\mathbf{v}_{j+1})_-} \longrightarrow 0
  \end{equation}
  for all $(\omega,t) \in [0,T] \times Z$, $\P$-a.s., uniformly with
  respect to $v_{j+1}$ in bounded sets of $H$.
  For any $h \in H$, the fundamental theorem of calculus yields
  \begin{align*}
    &\ip[\big]{D^jG(u_-+\varepsilon (v_{j+1})_-)\mathbf{v}_{j_-}%
    - D^jG(u_-)\mathbf{v}_{j-}}{h}\\
    &\hspace{3em} = \int_0^\varepsilon
      \ip[\big]{D^{j+1}G(u_-+s(v_{j+1})_-)(\mathbf{v}_{j+1})_-}{h}\,ds,
  \end{align*}
  hence, since $h$ is arbitrary, 
  \begin{align*}
    &\norm[\bigg]{\frac{1}{\varepsilon} %
      \bigl(D^jG(u_-+\varepsilon (v_{j+1})_-\mathbf{v}_{j-} %
      - D^jG(u_-)\mathbf{v}_{j-} \bigr)}\\
    &\hspace{3em} =
      \norm[\bigg]{\frac{1}{\varepsilon} \int_0^\varepsilon
      D^{j+1}G(u_-+s(v_{j+1})_-)(\mathbf{v}_{j+1})_-\,ds}\\
    &\hspace{3em} \lesssim (g_1+g_2)
      \bigl( 1 + \norm{u_-}^m + \norm{(v_{j+1})_-}^m\bigr)
      \norm{v_{1-}} \; \cdots \; \norm{v_{j-}} \, \norm{(v_{j+1})_-}\\
    &\hspace{3em} \leq (g_1+g_2) \biggl( \prod_{k=1}^{j+1} v_k^*
      + u^{*m} \prod_{k=1}^{j+1} v_k^*
      + v_{j+1}^{*(m+1)} \prod_{k=1}^j v_k^* \biggr)
  \end{align*}
  for any $\abs{\varepsilon} \leq 1$, where, as already done before,
  $g_1:=g_2:=g/2$ if $p \geq 2$. The left-hand side of \eqref{eq:kk}
  is thus dominated for all $(t,z) \in [0,T] \times Z$, $\P$-a.s.,
  modulo a constant, by the same expression appearing on the
  right-hand side of the previous inequality. This implies
  \begin{align*}
    &\norm[\Big]{\frac{1}{\varepsilon} %
      \bigl(D^jG(u_-+\varepsilon (v_{j+1})_-)\mathbf{v}_{j-} %
      - D^jG(u_-)\mathbf{v}_{j-} \bigr) %
      - D^{j+1}G(u_-)(\mathbf{v}_{j+1})_-}_{\bG^p}\\
    &\hspace{3em} \lesssim  \kappa(T)
      \biggl( \norm[\bigg]{ \prod_{k=1}^{j+1} v_k^* }_{L^p(\Omega)}
      + \norm[\bigg]{u^{*m} \prod_{k=1}^{j+1} v_k^*}_{L^p(\Omega)}
      + \norm[\bigg]{v_{j+1}^{*(m+1)} \prod_{k=1}^j v_k^*}_{L^p(\Omega)}
      \biggr)
  \end{align*}
  where, by H\"older's inequality,
  \begin{align*}
    \norm[\bigg]{ \prod_{k=1}^{j+1} v_k^* }_{L^p(\Omega)}
    &\leq \prod_{k=1}^{j+1} \norm[\big]{v_k}_{\bS^q} < \infty,\\
    \norm[\bigg]{u^{*m} \prod_{k=1}^{j+1} v_k^*}_{L^p(\Omega)}
    &\leq \norm[\big]{u}_{\bS^{mr}}^m
      \prod_{k=1}^{j+1} \norm[\big]{v_k}_{\bS^q} < \infty,\\
    \norm[\bigg]{v_{j+1}^{*(m+1)} \prod_{k=1}^j v_k^*}_{L^p(\Omega)}
    &\leq \norm[\big]{v_{j+1}}^{(m+1)}_{\bS^q}
      \prod_{k=1}^{j} \norm[\big]{v_k}_{\bS^q} < \infty.
  \end{align*}
  In fact, these three inequalities follow from
  \[
    \frac{j+1}{q} < \frac{1}{p}, \qquad
    \frac{1}{r} + \frac{j+1}{q} < \frac{1}{p}, \qquad
    \frac{m+1}{q} + \frac{j}{q} < \frac{1}{p},
  \]
  respectively, all of which are immediate consequences of the
  assumptions. The dominated convergence theorem
  thus yields
  \[
    \norm[\Big]{\frac{1}{\varepsilon} %
      \bigl(D^jG(u_-+\varepsilon (v_{j+1})_-)\mathbf{v}_{j-} %
      - D^jG(u_-)\mathbf{v}_{j-} \bigr) %
      - D^{j+1}G(u_-)(\mathbf{v}_{j+1})_-}_{\bG^p} \longrightarrow 0
  \]
  as $\varepsilon \to 0$.
  It remains to show that the convergence is uniform with respect to
  $v_1,\ldots,v_{j+1}$ bounded in $\bS^q$. To this end, we proceed as
  in the case $j=1$: for every measurable $E \in \cF$, the
  computations just carried out yield
  \begin{align*}
    &\norm[\Big]{1_E \Bigl(\frac{1}{\varepsilon} %
      \bigl(D^kG(u_-+\varepsilon (v_{k+1})_-)\mathbf{v}_{k-} %
      - D^kG(u_-)\mathbf{v}_{k-} \bigr) %
      - D^{k+1}G(u_-)(\mathbf{v}_{k+1})_- \Bigr)}_{\bG^p}\\
    &\hspace{3em} \lesssim \E 1_E \bigl(v_1^* \,\cdots\,v_{k+1}^* \bigr)^p
      + \E 1_E \bigl(u^{*m} v_1^* \,\cdots\, v_{k+1}^* \bigr)^p
      + \E 1_E \bigl( v_1^* \,\cdots\, v_k^* v_{k+1}^{*(m+1)} \bigr)^p,
  \end{align*}
  where the implicit constant depends on $\kappa(T)$.  Since
  $v_1,\ldots,v_{k+1}$ are bounded in $\bS^q$ and $k+1 \leq n$, the
  product $v_1^* \,\cdots\, v_{k+1}^*$ is bounded in
  $\bL^{q/n}$. Therefore, as $q/n>p$ by assumption, it follows that
  $\bigl(v_1^* \,\cdots\, v_{k+1}^*\bigr)^p$ is uniformly integrable.
  Similarly, defining $s$ by
  \[
    \frac{1}{s} := \frac{1}{r} + \frac{n}{q} < \frac{1}{p},
  \]
  H\"older's inequality yields
  \[
    \norm[\big]{u^{*m} v_1^* \,\cdots\, v_{j+1}^*}_{L^s(\Omega)}
    \leq \norm[\big]{u}^m_{\bS^{mr}} \norm[\big]{v_1}_{\bS^q}
    \; \cdots \; \norm[\big]{v_{j+1}}_{\bS^q},
  \]
  where the right-hand side is finite by assumption. Since $s>p$,
  $\bigl(u^{*m} v_1^* \,\cdots\, v_{k+1}^* \bigr)^p$ is
  uniformly integrable.
  Finally, defining $\ell$ by
  \[
    \frac1\ell := \frac{n-1}{q} + \frac{m+1}{q} = \frac{m+n}{q}
    < \frac{1}{p},
  \]
  H\"older's inequality yields, recalling that $j \leq n-1$,
  \begin{align*}
    \norm[\big]{v_1^* \,\cdots\, v_j^* v_{j+1}^{*(m+1)}}_{L^\ell(\Omega)}
    &\leq \norm[\big]{v_1}_{\bS^q} \; \cdots \; \norm[\big]{v_j}_{\bS^q}
      \norm[\big]{v_{j+1}^{*(m+1)}}_{L^{q/(m+1)}(\Omega)}\\
    &= \norm[\big]{v_1}_{\bS^q} \; \cdots \; \norm[\big]{v_j}_{\bS^q}
      \norm[\big]{v_{j+1}}^{m+1}_{\bS^q} < \infty,
  \end{align*}
  hence $\bigl(v_1^* \,\cdots\, v_j^* v_{j+1}^{*(m+1)} \bigr)^p$ is
  also uniformly integrable. One can now 
  choose the set $E$ and
  proceed exactly as in the proof of Theorem~\ref{thm:F1}
  for the case $j=1$ to conclude.
\end{proof}

The previous lemma implies, in particular, that
\begin{align*}
  f: \bS^q &\longrightarrow \bS^p,\\
  B: \bS^q &\longrightarrow L^p(\Omega;L^2(0,T;\cL^2(K,H))),\\
  G: \bS^q &\longrightarrow \bG^p
\end{align*}
are $n$ times Fr\'echet differentiable for every $q>(m+n)p$.  Indeed,
for any such $q$, one has $\frac mq + \frac nq=\frac{m+n}{q}<\frac1p$,
implying in particular that $\frac1p-\frac nq\in(0,1)$. 
Setting now $\frac1r:=(\frac mq)\vee\frac12(\frac1p-\frac nq)$, 
one has $r>1$,
$\frac{1}{r}+\frac{n}{q}<\frac1p$, and $\bS^q\subseteq\bS^{p\vee mr}$.\\
In fact, if $ \frac mq>\frac12(\frac1p-\frac nq)$ one has $\frac1r=\frac mq$,
from which $\frac{1}{r}+\frac{n}{q}=\frac mq + \frac nq<\frac1p$ and
$q=mr>p$, hence $\bS^q\subset \bS^{p\vee mr}$. If
$\frac mq\leq\frac12(\frac1p-\frac nq)$ one has 
$\frac1r=\frac12(\frac1p-\frac nq)<\frac1p-\frac nq$,
from which $\frac{1}{r}+\frac{n}{q}<\frac1p$,
and $q\geq mr$, hence $\bS^q\subseteq \bS^{p\vee mr}$.
The assertion follows then from Lemma~\ref{lm:Fcomp}.

\medskip

We can now state the main result of this section, as well as of the
whole paper.
\begin{thm}
  \label{thm:Fn}
  Let $n \geq 1$,
  \[
    q > \frac{(m+n)!}{(m+1)!} p.
  \]
  Assume \emph{(F$^{n}$)} and \emph{(A5$_r$)} for all $r\in[p,q]$.
  Then the solution map $(u_0 \mapsto u):\bL^q \to \bS^p$ is $n$ times
  Fr\'echet differentiable.
  Moreover, $D^nu(u_0) \in \cL_n(\bL^q;\bS^p)$ is the unique mild
  solution $\diff{n}{u}$ to
  \[
    d\diff{n}{u} + A\diff{n}{u}\,dt = \bigl( \Psi_n + Df(u)\diff{n}{u} \bigr)\,dt +
    \bigl( \Phi_n + DB(u)\diff{n}{u} \bigr)\,dW +\int_Z\bigl( \Theta_n +
    DG(u_-)\diff{n}{u}_- \bigr)\,d\bar{\mu}.
  \]
\end{thm}
\noindent Note that this equation is nothing else than \eqref{eq:xx},
and must be interpreted as the latter, i.e. in the sense of testing
against an $n$-tuple of vectors in $\bL^q$. Moreover, the initial
condition of the equation is the identity map if $n=1$, and zero if
$n \geq 2$.
\begin{proof}
  We shall assume, for simplicity, that $f=B=0$, as the argument in the
  general case $f \neq 0$, $B\neq 0$ is entirely analogous.  
  We are going to
  argue by induction on $\ell\in\{1,\ldots,n\}$.  The statement is true for $\ell=1$ by
  Theorem~\ref{thm:F1}. Now we assume that the statement is true for
  all $j \leq \ell-1$, $\ell\in\{2,\ldots,n\}$, and we prove it for $\ell$. Let $k \in
  \bL^q$, with
  $q > \frac{(m+\ell)!}{(m+1)!}p=(m+\ell)\ldots(m+2)p$.
  Thanks to Proposition~\ref{prop:aux} and the remarks following its
  proof, the equation
  \[
  d\diff{\ell}{u} + A\diff{\ell}{u}\,dt =  \int_Z \bigl( \Theta_{\ell} + DG(u_-)\diff{\ell}{u}_-
  \bigr)\,d\bar{\mu}, \qquad \diff{\ell}{u}(0)=0,
  \]
  admits a unique mild solution $\diff{\ell}{u} \in \cL_{\ell}(\bL^{q};\bS^p)$, because
  \[
  q > (m+\ell) \cdots (m+2)p \geq (\ell+m\ell-m)p.
  \]
  We are going to show that $\diff{\ell}{u} = D^{\ell}u(u_0)$ in
  $\cL_{\ell}(\bL^{q};\bS^p)$. Let $k \in \bL^{q}$: for brevity,
  we shall use the notation
  $\diff{\ell}{u}(u_0)k:=\diff{\ell}{u}(u_0)(k,\cdot, \ldots, \cdot)\in\cL_{\ell-1}(\bL^p,\bS^p)$
  and
  $\Theta_{\ell}(u_0)k:=\Theta_\ell(u_0)(k,\cdot, \ldots, \cdot)\in\cL_{\ell-1}(\bL^p,\bS^p)$.
    One has
    \begin{equation}
    \label{eq:ccc}
    \begin{split}
      &\frac{\diff{\ell-1}{u}(u_0+\varepsilon k) - \diff{\ell-1}{u}(u_0)}{\varepsilon}
      - \diff{\ell}{u}(u_0)k\\
      &\hspace{1.5em} = S \diamond_\mu \Bigl( \frac{\Theta_{\ell-1}(u_0+\varepsilon
        k) - \Theta_{\ell-1}(u_0)}{\varepsilon} - \Theta_{\ell}(u_0)k
      \Bigr)\\
      &\hspace{1.5em} \quad + S \diamond_\mu \Bigl(
      \frac{DG(u_-(u_0+\varepsilon k))\diff{\ell-1}{u}_-(u_0+\varepsilon k)%
        - DG(u_-(u_0))\diff{\ell-1}{u}_-(u_0)}{\varepsilon}\\
      &\hspace{2em} \phantom{\quad + S \diamond_\mu \Bigl(}
        - DG(u_-(u_0))\diff{\ell}{u}_-(u_0)k \Bigr),
    \end{split}
  \end{equation}
  where, by the inductive hypothesis, $\diff{\ell-1}{u}(u_0)=D^{\ell-1}u(u_0)$
  and $\diff{\ell-1}{u}(u_0+\varepsilon k)=D^{\ell-1}u(u_0+\varepsilon k)$ in
  $\cL_{\ell-1}(\bL^{q};\bS^p)$. We need to prove
  that the left-hand side of \eqref{eq:ccc} converges to zero as
  $\varepsilon \to 0$ in $\cL_{\ell-1}(\bL^{q},\bS^p)$ uniformly over $k$
  belonging to bounded sets of $\bL^{q}$.
  Thanks to \eqref{eq:sai}, one has
  \begin{align*}
    &\frac{\Theta_{\ell-1}(u_0+\varepsilon k) -
      \Theta_{\ell-1}(u_0)}{\varepsilon} - \Theta_{\ell}(u_0)k\\
    &\hspace{3em} = \frac{\Theta_{\ell-1}(u_0+\varepsilon k) -
      \Theta_{\ell-1}(u_0)}{\varepsilon} - D\Theta_{\ell-1}(u_0)k -
    D^2G(u_-)(u'_-,\diff{\ell-1}{u}_-)k,
  \end{align*}
  and we claim that
  \begin{equation}
    \label{eq:fino}
    \frac{\Theta_{\ell-1}(u_0+\varepsilon k) - \Theta_{\ell-1}(u_0)}{\varepsilon} -
    D\Theta_{\ell-1}(u_0)k \to 0
  \end{equation}
  in $\cL_{\ell-1}(\bL^{q};\bG^p)$ as
  $\varepsilon \to 0$, uniformly over $k$ belonging to bounded subsets
  of $\bL^{q}$. In fact, all terms in $\Theta_{\ell-1}$ are of the
  form
  \begin{equation}
    \label{eq:djb}
    D^jG(u_-(u_0))\bigl(D^{\alpha_1}u_-(u_0),\ldots,D^{\alpha_j}u_-(u_0)\bigr),
  \end{equation}
  with $j\leq \ell-1$ and $\alpha_i\geq 1$, $\sum \alpha_i=\ell-1$.
  Now, let $r>(\ell+m)p$ be
  such that $(u_0 \mapsto u): \bL^r \to \bS^r$ (which is possible
  because $q>(\ell+m)p$). Then $G: \bS^r \to
  \bG^p$ is $n$ times Fr\'echet
  differentiable by Lemma~\ref{lm:Fcomp}.
  Moreover, by the inductive hypothesis
  applied to $(u_0 \mapsto u) \in (\bL^r \to \bS^r)$, we have that $u_0
  \mapsto u$ is $\ell-1$ times Fr\'echet differentiable along $\bL^q$ if
  \[
  q > \frac{(m+\ell-1)!}{(m+1)!} r > \frac{(m+\ell-1)!}{(m+1)!} (m+\ell)p =
  \frac{(m+\ell)!}{(m+1)!} p.
  \]
  Therefore, if $q$ satisfies this condition, each term of the form
  \eqref{eq:djb} is Fr\'echet differentiable along $\bL^q$
  by the theorem on the Fr\'echet
  differentiability of composite functions (see for example \cite[Prop.~1.4]{AmbPro}).
  Hence, \eqref{eq:fino} is indeed true, and the expression within
  parentheses in the first term on the right-hand side of
  \eqref{eq:ccc} converges to
  \[
    - D^2G(u_-(u_0))\bigl( Du_-(u_0),D^{\ell-1}u_-(u_0) \bigr)k
  \]
  in $\cL_{\ell-1}(\bL^q;\bG^p)$ as
  $\varepsilon \to 0$, uniformly over $k$ belonging to bounded subsets
  of $\bL^q$.
  Let us now consider the second term on the right-hand side of
  \eqref{eq:ccc}. One has, recalling that $v_j=D^j u$
  for every $j\leq \ell-1$ by inductive hypothesis,
  \begin{align*}
    &\frac{DG(u_-(u_0+\varepsilon k))\diff{\ell-1}{u}_-(u_0+\varepsilon k)%
        - DG(u_-(u_0))\diff{\ell-1}{u}_-(u_0)}{\varepsilon} %
      - DG(u_-(u_0))\diff{\ell}{u}_-(u_0)k \\
    &\hspace{3em} = DG(u_-(u_0)) \biggl(
    \frac{D^{\ell-1}u_-(u_0+\varepsilon k) - D^{\ell-1}u_-(u_0)}{\varepsilon}
    - \diff{\ell}{u}_-(u_0)k \biggr)\\
    &\hspace{3em} \quad + \frac{DG(u_-(u_0+\varepsilon
      k))%
      - DG(u_-(u_0))}{\varepsilon}D^{\ell-1}u_-(u_0+\varepsilon k),
  \end{align*}
  where the second term on the right-hand side converges to
  \[
    D^2G(u_-(u_0)) \bigl( Du_-(u_0),D^{\ell-1}u_-(u_0) \bigr)k
  \]
  in $\cL_{\ell-1}(\bL^q;\bG^p)$ as
  $\varepsilon \to 0$, uniformly over $k$ bounded in $\bL^{q}$,
  because again everything depends only on derivatives of order at
  most $\ell-1$ and we can apply the usual criteria on Fr\'echet
  differentiability of multilinear maps and composite functions. Note
  that this term cancels out with the corresponding one obtained
  previously. 
  Going back then to \eqref{eq:ccc}, testing 
  by an arbitrary element $(k_2,\ldots,k_\ell)\in (\bL^q)^{\ell-1}$,
  and using Lemma~\ref{prop:scs}, 
  we infer that
  \begin{align*}
  &\left\|\left(\frac{\diff{\ell-1}{u}(u_0+\varepsilon k)(k_2,\ldots,k_\ell) 
  - \diff{\ell-1}{u}(u_0)(k_2,\ldots,k_\ell)}{\varepsilon}
      - \diff{\ell}{u}(u_0)(k,k_2,\ldots,k_\ell)\right)
  \right\|_{\bS^p}\\
  &\lesssim
  \left\|
  \left(\frac{\Theta_{\ell-1}(u_0+\varepsilon k)(k_2,\ldots,k_\ell)
   - \Theta_{\ell-1}(u_0)(k_2,\ldots,k_\ell)}{\varepsilon}
   - \Theta_{\ell}(u_0)(k, k_2,\ldots,k_\ell\right)\right.\\
   &\qquad\qquad\left.+ \frac{DG(u_-(u_0+\varepsilon k))%
      - DG(u_-(u_0))}{\varepsilon}D^{\ell-1}u_-(u_0+\varepsilon k)(k_2,\ldots,k_\ell)\right\|_{\bG^p}\\
   &+\left\|DG(u_-(u_0)) \biggl(
    \frac{D^{\ell-1}u_-(u_0+\varepsilon k) - D^{\ell-1}u_-(u_0)}{\varepsilon}(k_2,\ldots,k_\ell)
    - \diff{\ell}{u}_-(u_0)(k,k_2,\ldots,k_\ell) \biggr)
   \right\|_{\bG^p}.
  \end{align*}
  Taking supremum over $(k_2,\ldots,k_\ell)$ bounded in $(\bL^q)^{\ell-1}$
  and using the Lipschitz-continuity of $G$,
  we infer that, for every $T_0\in\mathopen(0,T\mathclose]$,
   \begin{align*}
  &\left\|\left(\frac{\diff{\ell-1}{u}(u_0+\varepsilon k)
  - \diff{\ell-1}{u}(u_0)}{\varepsilon}
      - \diff{\ell}{u}(u_0)k\right)
  \right\|_{\cL_{\ell-1}(\bL^q;\bS^p(0,T_0))}\\
  &\lesssim
  \left\|
  \left(\frac{\Theta_{\ell-1}(u_0+\varepsilon k)
   - \Theta_{\ell-1}(u_0)}{\varepsilon}
   - \Theta_{\ell}(u_0)k\right)\right.\\
   &\qquad\qquad\left.+ \frac{DG(u_-(u_0+\varepsilon k))%
      - DG(u_-(u_0))}{\varepsilon}D^{\ell-1}u_-(u_0+\varepsilon k)\right\|_{\cL_{\ell-1}(\bL^q;\bG^p(0,T_0))}\\
   &+\kappa(T_0)\left\|\frac{D^{\ell-1}u_-(u_0+\varepsilon k) - D^{\ell-1}u_-(u_0)}{\varepsilon}
    - \diff{\ell}{u}_-(u_0)k \right\|_{\cL_{\ell-1}(\bL^q;\bG^p(0,T_0))}.
  \end{align*}
  By the continuity of $\kappa$ we can choose $T_0$ sufficiently small
  such that, after rearranging terms,
   \begin{align*}
  &\left\|\left(\frac{\diff{\ell-1}{u}(u_0+\varepsilon k)
  - \diff{\ell-1}{u}(u_0)}{\varepsilon}
      - \diff{\ell}{u}(u_0)k\right)
  \right\|_{\cL_{\ell-1}(\bL^q;\bS^p(0,T_0))}\\
  &\lesssim
  \left\|
  \left(\frac{\Theta_{\ell-1}(u_0+\varepsilon k)
   - \Theta_{\ell-1}(u_0)}{\varepsilon}
   - \Theta_{\ell}(u_0)k\right)\right.\\
   &\qquad\qquad\left.+ \frac{DG(u_-(u_0+\varepsilon k))%
      - DG(u_-(u_0))}{\varepsilon}D^{\ell-1}u_-(u_0+\varepsilon k)\right\|_{\cL_{\ell-1}(\bL^q;\bG^p(0,T_0))}.
  \end{align*}
  Using the same argument leading to \eqref{eq:ecce} in the proof of
  Proposition~\ref{prop:aux}, a classical patching argument yields then
  \begin{align*}
  &\left\|\left(\frac{\diff{\ell-1}{u}(u_0+\varepsilon k)
  - \diff{\ell-1}{u}(u_0)}{\varepsilon}
      - \diff{\ell}{u}(u_0)k\right)
  \right\|_{\cL_{\ell-1}(\bL^q;\bS^p)}\\
  &\lesssim
  \left\|
  \left(\frac{\Theta_{\ell-1}(u_0+\varepsilon k)
   - \Theta_{\ell-1}(u_0)}{\varepsilon}
   - \Theta_{\ell}(u_0)k\right)\right.\\
   &\qquad\qquad\left.+ \frac{DG(u_-(u_0+\varepsilon k))%
      - DG(u_-(u_0))}{\varepsilon}D^{\ell-1}u_-(u_0+\varepsilon k)\right\|_{\cL_{\ell-1}(\bL^q;\bG^p)}
  \end{align*}
  on the whole time interval $[0,T]$. Taking into account the remarks made above
  we have that
   \begin{align*}
  &\lim_{\eps\to0^+}\sup_{\norm{k}_{\bL^q}\leq 1}
  \left\|\left(\frac{\diff{\ell-1}{u}(u_0+\varepsilon k)
  - \diff{\ell-1}{u}(u_0)}{\varepsilon}
      - \diff{\ell}{u}(u_0)k\right)
  \right\|_{\cL_{\ell-1}(\bL^q;\bS^p)}\\
  &\lesssim\lim_{\eps\to0^+}\sup_{\norm{k}_{\bL^q}\leq 1}
  \left\|
  \left(\frac{\Theta_{\ell-1}(u_0+\varepsilon k)
   - \Theta_{\ell-1}(u_0)}{\varepsilon}
   - \Theta_{\ell}(u_0)k\right)\right.\\
   &\qquad\qquad\qquad\left.+ \frac{DG(u_-(u_0+\varepsilon k))%
      - DG(u_-(u_0))}{\varepsilon}D^{\ell-1}u_-(u_0+\varepsilon k)\right\|_{\cL_{\ell-1}(\bL^q;\bG^p)}\\
   &=\left\|-D^2G(u_-(u_0)) \bigl( Du_-(u_0),D^{\ell-1}u_-(u_0) \bigr)k\right.\\
   &\qquad\left.+D^2G(u_-(u_0)) \bigl( Du_-(u_0),D^{\ell-1}u_-(u_0) \bigr)k
   \right\|_{\cL_{\ell-1}(\bL^q;\bG^p)}\\
   &=0.
  \end{align*}
  We conclude that the left-hand side of
  converges to zero in $\cL_{\ell-1}(\bL^q;\bS^p)$ as
  $\varepsilon \to 0$, uniformly with respect to $k$ belonging to any
  bounded subset of $\bL^q$, as required.
\end{proof}


\ifbozza\newpage\else\fi
\bibliographystyle{amsplain}
\bibliography{ref,extra}

\end{document}